\documentclass[12pt,reqno]{amsart}
\usepackage{amsmath}
\usepackage{amssymb}
\usepackage{amstext}
\usepackage{mathrsfs}
\usepackage{a4wide}
\usepackage{graphicx}
\usepackage{bbm}
\usepackage{verbatim}
\usepackage{bm}

\allowdisplaybreaks \numberwithin{equation}{section}
\usepackage{color}
\usepackage{cases}

\numberwithin{equation}{section}

\newtheorem{theorem}{Theorem}[section]
\newtheorem{proposition}[theorem]{Proposition}
\newtheorem{corollary}[theorem]{Corollary}
\newtheorem{lemma}[theorem]{Lemma}

\theoremstyle{definition}

\theoremstyle{remark}
\newtheorem{remark}[theorem]{Remark}

\begin{document}		
	\title
	[vortex sheet in domain]{Existence of stationary vortex sheets for the 2D Euler equation}
	
	\author{Daomin Cao, Guolin Qin, Changjun Zou}
	
	\address{Institute of Applied Mathematics, Chinese Academy of Sciences, Beijing 100190, and University of Chinese Academy of Sciences, Beijing 100049,  P.R. China}
	\email{dmcao@amt.ac.cn}
	
	\address{Institute of Applied Mathematics, Chinese Academy of Sciences, Beijing 100190, and University of Chinese Academy of Sciences, Beijing 100049,  P.R. China}
	\email{qinguolin18@mails.ucas.edu.cn}
	
	\address{Institute of Applied Mathematics, Chinese Academy of Sciences, Beijing 100190, and University of Chinese Academy of Sciences, Beijing 100049,  P.R. China}
	\email{zouchangjun17@mails.ucas.ac.cn}
	
	\thanks{This work was supported by NNSF of China Grant 11831009 and Chinese Academy of Sciences (No. QYZDJ-SSW-SYS021).}

	\begin{abstract}
           We investigate a steady planar flow of an ideal fluid in a (bounded or unbounded) domain  $\Omega\subset \mathbb{R}^2$. Let $\kappa_i\not=0$, $i=1,\ldots, m$, be $m$ arbitrary fixed constants. For any given non-degenerate critical point $\mathbf{x}_0=(x_{0,1},\ldots,x_{0,m})$ of the Kirchhoff-Routh function defined on $\Omega^m$ corresponding to $(\kappa_1,\ldots, \kappa_m)$,  we construct a family of stationary planar flows with  vortex sheets that have large vorticity amplitude and are perturbations of small circles centered near $x_i$, $i=1,\ldots,m$.  The proof is accomplished via the implicit function theorem with suitable choice of function spaces. This seems to be the first nontrivial result on the existence of stationary vortex sheets in domains.
	\end{abstract}
	
	\maketitle{\small{\bf Keywords:} Euler equation, vortex sheets, non-degenerate, the Birkhoff-Rott operator, implicit function theorem. \\
		
	}

\section{Introduction}
Let $\Omega\subset\mathbb{R}^2$ be a bounded or unbounded domain. We consider the stationary Euler equation
\begin{equation}\label{1-1}
	\begin{cases}
		\mathbf{u}\cdot \nabla 	\mathbf{u}=-\nabla P&\quad \text{in}\, \Omega,\\
		\nabla\cdot 	\mathbf{u}=0&\quad \text{in}\, \Omega,\\
		\mathbf{u}\cdot \nu=0 &\quad \text{on}\, \partial\Omega,
	\end{cases}
\end{equation}
where $\mathbf{u}=(u_1,u_2)$ is the velocity field, $P$ is the scalar pressure and $\nu$ is the outward unit normal of $\partial \Omega$.

In a planar flow, the vorticity is defined by  the curl of the velocity field, this is,  $\omega=curl \mathbf{u}:=\partial_1u_2-\partial_2u_1$. Taking the curl of the equation \eqref{1-1}, we find that $\omega$ satisfies the following vorticity equation
\begin{equation}\label{1-2}
	\mathbf{u}\cdot \nabla \omega =0\quad\text{in}\,\Omega.
\end{equation}
The velocity is recovered by the Biot-Savart law
$$
\ \mathbf{u}=\nabla^\perp(-\Delta)^{-1}\omega,
$$
where $(x_1,x_2)^\perp=(x_2,-x_1)$ and the operator $(-\Delta)^{-1}$ is given by
\begin{equation*}
	(-\Delta)^{-1}\omega(x)=\int_{\mathbb{R}^2}G(x- y)\omega(y)dy.
\end{equation*}
Here $G(x,y)=\frac{1}{2\pi}\ln \frac{1}{|x-y|}-H(x,y)$ is the Green function of $-\Delta$ in $\Omega$. We denote $\psi=(-\Delta)^{-1}\omega$ to be the stream function, then the velocity field can be derived by $\mathbf{u}=\nabla^\perp \psi$.

In the last century, the two-dimensional Euler equation has been intensively studied and the global well-posedness of the vorticity equation with initial data in $L^1\cap L^\infty$ was proved by Yudovich in the classical paper \cite{Yud}. However, many physical phenomena possess strong and irregular fluctuations, such as fluids with small viscosity, where flows tend to separate from rigid walls and sharp corners \cite{Bir, MB}. To model this phenomenon mathematically, the most natural way is to think of a solution to the Euler equation, in which the velocity changes sign discontinuously across a stream line. This discontinuity induces vorticity concentrated on a curve, which is only a measure rather than a bounded function.

A  velocity discontinuity in an inviscid 2D flow is called a vortex sheet, whose vorticity concentrates as a measure along a curve. Suppose that $\omega$ is a weak solution to the Euler equation concentrated on a finite number of closed curves $\Gamma_i$ parameterized by $z_i(\theta)$.  This is, for any test function $\phi\in C_c^\infty(\Omega)$, $\omega$ is a measure such that
$$\int_\Omega \phi(x)d\omega(x)=\sum_i \int  \gamma_i(\alpha)\phi(z_i(\alpha))|z_i'(\alpha)|d\alpha,$$
where $\gamma_i(\alpha)$ is the vorticity strength at $z_i(\alpha)$. Then the equation of the sheet can be derived by the Birkhoff-Rott operator in a domain \cite{Bir2, GPSY1, LNS, MB, Rot}
\begin{equation}\label{1-3}
	BR(z,\gamma)(x):=\frac{1}{2\pi}P.V. \int \frac{(x-z(\alpha))^\perp}{|x-z(\alpha)|^2} \gamma(\alpha)|z'(\alpha)| d\alpha+\int \nabla^\perp H(x,z(\alpha))\gamma(\alpha)|z'(\alpha)| d\alpha,
\end{equation}
where $P.V.$ stands for principal value integral. \eqref{1-3} yields the motion of the sheet
\begin{equation}\label{1-4}
	\mathbf{u}(z_i(\theta))=-BR(z_i(\theta))
\end{equation}
with $BR(z_i(\theta)):=-\sum_j BR(z_j,\gamma_j)(z_i(\theta))$.

Significant efforts have been made in mathematical study of the theory of vortex sheet. In the elegant paper \cite{Del}, Delort proved global existence of weak solutions with an initial $L^2_{\text{loc}}$ velocity and a positive measure vorticity. Later, the proof was simplified by Majda \cite{Maj}.  Duchon and Robert\cite{DR}
 established global existence for a class of initial data concentrated closed to a line. Existence in different setting of vortex sheet with a distinguished sign was also  obtained in \cite{EM, Sch}. For vorticity without a definite sign, only partial results on the existence are known under some additional assumptions \cite{LNX, SSBF, W}. Note that uniqueness for such solutions still remains open.

The operator \eqref{1-3} is so singular that blow up may occur in the motion of vortex sheet. Indeed, singular formulation was conjectured by Birkhoff \cite{Bir2}, and by Birkhoff and Fisher \cite{BF}. In \cite{Moo}, Moore showed the possibility that the curvature blows up in finite time even though the initial data is analytic. Moore's  result was also supported by numerical study \cite{Kra}. Ill-posedness for vortex sheet problem in the space $H^s$ with $s>\frac{3}{2}$ was obtained by Caflisch and Orellana \cite{CO}. These results demonstrate that the study of vortex sheet is extremely delicate, and hence exact solutions, in particular relative equilibria, are of great importance since their structures persist for long time.

Nevertheless, very few relative equilibria are known.  For the vortex sheets in $\mathbb{R}^2$, except for circles and lines, the only nontrivial examples include : uniformly rotating segment \cite{Bat},
in which the vorticity is supported on a segment of length $2a$ with density
\begin{equation*}
	\gamma(x)=\Omega\sqrt{a^2-x^2}, \ \ \ \ \ \ \text{for} \ \ x\in[-a,a]
\end{equation*}
and angular velocity $\Omega$. A generalization of the rotating segment is the Protas-Sakajo class \cite{PS}, which is  made out of segments rotating about a common center of rotation with endpoints at the vertices of a regular polygon. Recently, a new class of  vortex sheet was obtained in  \cite{GPSY2} via degenerate  bifurcation from rotating circles. Note that the existence of nontrivial steady vortex sheet in $\mathbb R^2$ is not apparent in view of the rigidity results obtained in \cite{GPSY1}, where the authors showed for uniformly rotating vortex sheets with angular velocity $\Omega\leq 0$ and strength $\gamma>0$, only trivial solutions exist.

\emph{In a domain $\Omega\subsetneq \mathbb{R}^2$,  there seems no stationary vortex sheet is known so far}.

The purpose of the present  paper is to construct a family of stationary vortex sheets for a domain (bounded or not), whenever the Kirchhoff-Routh function possesses nondegenerate critical points. This is the first result on this issue.

For any given integer $m>0$, and $m$ real numbers $\kappa_1,\kappa_2,\ldots,\kappa_m$, define the  Kirchhoff-Routh function on $\Omega^m =\{\mathbf{x}=(x_1,x_2,..,x_m)\,\mid\, x_i\in\Omega,\text{for}\,i=1,\ldots,m\}$
as follows
\begin{equation}\label{1-5}
	\mathcal{W}_m(x_1,x_2,..,x_m)=-\sum\limits_{i\neq j}^m\kappa_i\kappa_j G(x_i,x_j)+\sum\limits_{i=1}^k\kappa_i^2H(x_i,x_i).
\end{equation}
It is known that the location of $m$-point vortices with strength $\kappa_i$ ($i=1,\ldots,m$) in $\Omega$ must be a critical point of $\mathcal{W}_m$, see e.g. \cite{Lin1, Lin2}. Results on the existence and non-degeneracy of critical points for W can be found in \cite{BP, BPW}. In \cite{GT}, it
was proved that if $\Omega$ is convex then there is no critical point of $\mathcal{W}_m$ in $\Omega ^m$ with $m\geq 2$ provided $\kappa_i>0$ for all $i = 1,\ldots, m$. Let us point out that although the non-degeneracy of critical points for the
Kirchhoff-Routh functions in a general domain is not an easy issue, it is true for most of the domains, as
proved in \cite{Bar, MP}. On the other hand, in \cite{CF}, it shows that in a convex domain, $\mathcal{W}_1$ has a unique critical
point, which is also non-degenerate. In a recent paper \cite{CYY}, the first author, Yan and Yu obtained some existence and nondegeneracy results on critical points of the Kirchhoff-Routh function on unbounded domains.

Giving point $\mathbf{x}_0=(x_{0,1},\ldots,x_{0,m})\in \Omega^m$ that is a nondegenerate critical point of $\mathcal{W}_m$, for $\varepsilon$ small, we will construct  a branch of vortex sheets concentrated on a finite number of closed curves $\Gamma_i$. Moreover, each $\Gamma_i$   is the perturbation of a small circle with radius $\varepsilon$ centered at some point $x_{\varepsilon,\tau, i}\in \Omega$ close to $x_{0,i}$. The vorticity of $\omega$ satisfies $\int_0^{2\pi} \gamma_i(\alpha)|z_i'(\alpha)| d\alpha \to \kappa_i\not=0$. This result shows the rich diversity of stationary vortex sheet solutions despite that the well-posedness is not fully 
understood.

Our main theorem is  as follows.
\begin{theorem}\label{thm1}
	Let $\Omega\subset \mathbb{R}^2$ be a domain (may be unbounded), $\kappa_i\not=0$ ($i=1,\ldots, m)$ be $m$ given number. Suppose that $\mathbf{x}_0=(x_{0,1},\ldots,x_{0,m})\in \Omega^m$ with $x_{0,i}\not=x_{0,j}$ for $i\not=j$ is an isolated critical point of $\mathcal{W}_m$ defined by \eqref{1-5} satisfying the nondegeneracy condition: $\text{deg} \left(\nabla \mathcal{W}_m, \mathbf{x}_0\right)\not=0$. Then, there are $\varepsilon_0>0$ and $\tau_0>0$, such that for all $0<\varepsilon<\varepsilon_0$ and $-\tau_0<\tau<\tau_0$, there exists a stationary vortex sheet $\omega_{\varepsilon,\tau}$ possessing the following properties:
	\begin{itemize}
		\item[(i)] $\omega_{\varepsilon,\tau}=\sum\limits_{i=1}^m\gamma_i \delta_{\Gamma_i}$ concentrates on a finite number of closed curves $\Gamma_i$ with strength $\gamma_i$. Moreover, it holds that $\gamma_i=\frac{\kappa_i+O(\varepsilon)}{2\pi\varepsilon}$ and each $\Gamma_i$ is a perturbation of a small circle with radius $\varepsilon$ and centered at some point $x_{\varepsilon,\tau, i}\in \Omega$ satisfying $|x_{\varepsilon,\tau,i}-x_{0,i}|=O(\varepsilon)$.
		\item[(ii)]
		As $\varepsilon\to 0^+$ and $\tau\to 0$, one has in the sense of measure
		\begin{equation*}
			\omega_{\varepsilon,\tau}\to \sum\limits_{i=1}^{m}\kappa_i\delta(x-x_{0,i})\,\,\,\text{weakly},
		\end{equation*}
where, $\delta(x-x_{0,i})$ is the Dirac delta function concentrated at the point $x_{0,i}$.
		\item[(iii)]
		For any $i=1,\ldots,m$, the interior of $\Gamma_i$ is convex.
	\end{itemize}
\end{theorem}
\begin{remark}\label{rem1}
	Our result does not rely on the sign of $\kappa_i$, which is essential in the global existence of the initial problem as mentioned above.
\end{remark}
\begin{remark}\label{rem2}
	For simplicity, all the scales of $\Gamma_i$ ($i=1,\ldots,m$) are chosen to be of the same order. But this is not necessary and one may construct vortex sheet concentrated on $\Gamma_i$ with the  scale of $\varepsilon_i$ ($i=1,\ldots,m$) different from each other.
\end{remark}

Fixing $\tau\in (-\tau_0, \tau_0)$, say $\tau=0$, we obtain a family of vortex sheet solutions parameterized by $\varepsilon$, which is of special interest.
\begin{corollary}\label{thm2}
	Let $\Omega\subset \mathbb{R}^2$ be a domain (may be unbounded), $\kappa_i\not=0$ ($i=1,\ldots, m)$ be $m$ given numbers. Suppose that $\mathbf{x}_0=(x_{0,1},\ldots,x_{0,m})\in \Omega^m$ with $x_{0,i}\not=x_{0,j}$ for $i\not=j$  is an isolated critical point of $\mathcal{W}_m$ defined by \eqref{1-5} satisfying the nondegeneracy condition: $\text{deg} \left(\nabla \mathcal{W}_m, \mathbf{x}_0\right)\not=0$. Then, there is an $\varepsilon_0>0$, such that for all $0<\varepsilon<\varepsilon_0$, there exists a stationary vortex sheet $\omega_\varepsilon$ possessing the following properties:
	\begin{itemize}
		\item[(i)] $\omega_\varepsilon=\sum\limits_{i=1}^m\gamma_i \delta_{\Gamma_i}$ concentrates on a finite number of closed curves $\Gamma_i$ with strength $\gamma_i$. Moreover, it holds that $\gamma_i=\frac{\kappa_i+O(\varepsilon)}{2\pi\varepsilon}$ and each $\Gamma_i$ is a perturbation of a small circle with radius $\varepsilon$ and centered at some point $x_{i,\varepsilon}\in \Omega$ satisfying $|x_{\varepsilon,i}-x_{0,i}|=O(\varepsilon)$.
		\item[(ii)]
		As $\varepsilon\to 0^+$, one has in the sense of measure
		\begin{equation*}
			\omega_{\varepsilon}\to \sum\limits_{i=1}^{m}\kappa_i\delta(x-x_{0,i})\,\,\, \text{wealy},
		\end{equation*}
where, $\delta(x-x_{0,i})$ is the Dirac delta function concentrated at the point $x_{0,i}$.
		\item[(iii)]
		For any $i=1,\ldots,m$, the interior of $\Gamma_i$ is convex.
	\end{itemize}
\end{corollary}
The result of Corollary \ref{thm2} is closely related to the classical problem: regularization of point vortices
for the Euler equation, which means justifying the weak formulation for
point vortex solutions of the incompressible Euler equations by approximating
these solutions with more regular solutions. In fact, the  vortex sheets obtained in Corollary \ref{thm2} belong to the space $H^{-1}(\mathbb R^2)$, while the point vortices solution belongs to $H^{-1-\sigma}(\mathbb R^2)$ for any $\sigma>0$, which is more singular than a vortex sheet. Thus,  our result can be regarded as a desingularization of point vortices in some way. For more literature on  desingularization  problem, we refer to \cite{Cao3, CWZ, Dav, HM} and references therein.

Next we shall sketch the basic ideas used to prove the main result. Thanks to Lemma 2.1 in\ \cite{GPSY1},  we are able to formulate the conditions that the Birkhoff-Rott integral \eqref{1-3} satisfies for a stationary vortex sheet as a system of  $2m$ coupled integro-differential  equations $\mathcal{F}_{i,1}(\varepsilon, \mathbf{x}, \bm{f}, \bm{g})=0$ and $\mathcal{F}_{i,2}(\varepsilon, \mathbf{x}, \bm{f}, \bm{g})=0$, $i=1,\ldots,m$. We expect that the case $(\varepsilon, \mathbf{x},\bm{f}, \bm{g})=(0,\mathbf{x}_0,0,0)$ corresponds to the point vortices and hence $(0, \mathbf{x}_0, 0, 0)$ is a solution to $\mathcal{F}_{i,1}=0$ and $\mathcal{F}_{i,2}=0$ provided that $\mathbf{x}_0$ is a critical point of $\mathcal{W}_m$. Therefore the  first step is to extend $\mathcal{F}_{i,1}$ and  $\mathcal{F}_{i,2}$ such that $\varepsilon\leq 0$ is allowed. Then one can verify  that $\mathcal{F}_{i,1} (0, \mathbf{x}_0, 0, 0)=\mathcal{F}_{i,2}(0, \mathbf{x}_0, 0, 0)=0$ does hold when $ \mathbf{x}_0$ is a critical point of $\mathcal{W}_m$, and hence we obtain a trivial solution. To apply the implicit function theorem at the solution $(0, \mathbf{x}, 0, 0)$, the Gateaux derivative of $\bm{\mathcal {F}}:=(\mathcal{F}_{1,1}, \mathcal{F}_{1,2}, \ldots \mathcal{F}_{m,1}, \mathcal{F}_{m,2})$ should be an isomorphism, which unfortunately fails. Detailed calculations show that $D\bm{\mathcal F}$ has a $2m$ dimensional kernel $\prod_{i=1}^m \text{span}\{(a_1\cos(\theta)+b_1\sin(\theta), \kappa_i(a_1\cos(\theta)+b_1\sin(\theta)))\}$. Hence, we have to consider the equations in quotient spaces and impose the conditions $-\kappa_i \int \mathcal{F}_{i,1} \sin(\theta)d\theta=\int \mathcal{F}_{i,2} \cos(\theta)d\theta$ and $\kappa_i \int \mathcal{F}_{i,1}\cos(\theta)d\theta=\int \mathcal{F}_{i,2} \sin(\theta)d\theta$ for all $i=1,\ldots,m$. Although these conditions seem to be complicated, we successfully  convert  them into a concise equation $\nabla \mathcal{ W}_m(\mathbf{x})=O(\varepsilon)$, which is solvable near $\mathbf{x}_0$ due to the nondegeneracy of $\nabla \mathcal{W}_m$ at $\mathbf{x}_0$. Finally, we can apply the implicit function theorem to obtain the existence.  The convexity of the interior of $\Gamma_i$ follows  from calculating the curvature directly. We point out that our procedure of proving Theorem \ref{thm1} borrows the idea of  Lyapunov-Schmidt reduction and  local bifurcation theory.

The  ideas and methods introduced in the present paper may be widely applied to a variety of situations and other models. For example, one may consider an ideal fluid with an irrotational background flow $u_0=\nabla^\perp \psi_0$, where $\psi_0$ is a given harmonic function. In this case, the Kirchhoff-Routh function is given by (see \cite{CYY})
\begin{equation*}
	\mathcal{W}_{m, \psi_0}(x_1,x_2,..,x_m)=-\sum\limits_{i\neq j}^m\kappa_i\kappa_j G(x_i,x_j)+\sum\limits_{i=1}^k\kappa_i^2H(x_i,x_i)+2\sum\limits_{i=1}^k\kappa_i \psi_0(x_i).
\end{equation*}
Although $\mathcal{W}_{m, \psi_0}$ is slightly different from $\mathcal{W}_{m}$ given by \eqref{1-5}, 
we believe that  our method can be modified to construct vortex sheets near critical points of $\mathcal{W}_{m, \psi_0}$ in this situation.  In a recent work, the authors modified the method of this paper and contructed water waves of both finite depth and infinite depth with vortex sheets \cite{CQZ}.

We would like to make a brief remark on the approach of constructing vortex patches via the contour dynamics equation, which shares a similar spirit as the construction of vortex sheet  we consider. Many celebrated contributions have been made with the contour dynamics equation method in recent years, see e.g. \cite{Cas1, Cas2, Cas3, Cas4, de3,de2, Has3, HM,Hmi,Hmi2} and references therein. However, since a vortex patch is actually a bounded function, the contour dynamics equation is more regular than the equations of a vortex sheet. Hence, more attentions are needed to pay in the process of our proof.

This paper is organized as follows. In section 2, we derive the  equations that the Birkhoff-Rott integral satisfies for a stationary vortex sheet and define the function spaces which will be used later. In Section 3, we extend the functionals and show their $C^1$ regularity.  Section 4 is devoted to study the linearization operators, where we prove that the derivative is an isomorphism in  quotient spaces. In Section 5, we adjust $\mathbf{x}_{\varepsilon}$ such that the range of our functional belongs to the quotient spaces and apply the implicit function theorem to prove Theorem \ref{thm1}.

\section{Formulation and functional setting}\label{2}
Since $\omega$ is a stationary sheet, using Lemma 2.1 in \cite{GPSY1}, we derive the following equations that the BR equation \eqref{1-4} and  $\gamma_i$ satisfy.
\begin{equation}\label{2-1}
	BR(z_i(\theta))\cdot \mathbf{n}(z_i(\theta))=0,
\end{equation}
where $\mathbf{n}(z_i(\theta))$ is the unite normal vector of $\Gamma_i$ at $z_i(\theta)$,
and
\begin{equation}\label{2-2}
	BR(z_i(\theta))\cdot \mathbf{s}(z_i(\theta))\frac{\gamma_i(\theta)}{|z_i'(\theta)|}=C,
\end{equation}
where $\mathbf{s}(z_i(\theta))$ is the unite tangential vector.
Note that \eqref{2-2} can be rewritten as
\begin{equation}\label{2-3}
	(I-P_0)\left[BR(z_i(\theta))\cdot \mathbf{s}(z_i(\theta))\frac{\gamma_i(\theta)}{|z_i'(\theta)|}\right]=0,
\end{equation}
where $P_0$ is the projection to the mean defined by $P_0 f:=\frac{1}{2\pi}\int_0^{2\pi} f(\theta)d\theta$.  Since $\mathbf{x}_0\in \Omega^m$, we can take $r_0>0$ sufficiently small such that $B_{r_0}(x_{0,i})\subset \Omega$ for all $i=1,\ldots,m$, where $B_{r_0}(x_{0,i})$ is the ball with radius $r_0$ and centered at $x_{0,i}$. We aim to construct vortex sheets localized near $\mathbf{x}_0$. Thus, for $\varepsilon>0$ small,
we assume that $z_i$($i=1,\ldots, m$) are of the following form
$$z_i(\theta)=x_{i}+\varepsilon R_i(\theta)(\cos\theta, \sin\theta)$$
with $R_i(\theta)=1+\varepsilon f_i(\theta),$ and  $x_i\in B_{r_0}(x_{0,i})$  to be chosen later.
We also assume $$\gamma_i(\alpha)=\frac{\kappa_i+\varepsilon g_i(\alpha)}{2\pi |z'_i(\alpha)|}.$$

We end this section by introducing some notations and definitions that will be used in this paper and reformulating equations \eqref{2-1} and \eqref{2-3}. Denote the mean value of integral of $g$ on the unit circle by
$$\int\!\!\!\!\!\!\!\!\!\; {}-{} g(\tau)d\tau:=\frac{1}{2\pi}\int_0^{2\pi}g(\tau)d\tau$$
and  set
\begin{equation*}
	\begin{split}
		&A(\theta, \alpha):=4\sin^2\left(\frac{\theta-\alpha}{2}\right),\,\,\,A_{ij}=|x_{i}-x_{j}|^2,\\
		&B(f,\theta, \alpha):=4(f(\theta)+f(\alpha))\sin^2\left(\frac{\theta-\alpha}{2}\right)+\varepsilon\left((f(\theta)-f(\alpha))^2+4 f(\theta)f(\alpha)\sin^2\left(\frac{\theta-\alpha}{2}\right)\right),\\
		&B_{ij}(\theta, \alpha)=2(x_{i}-x_{j})\cdot((\cos\theta, \sin\theta)-(\cos\alpha, \sin\alpha))+2\varepsilon(x_{i}-x_{j})\cdot\left(f_i(\theta)(\cos\theta, \sin\theta)\right.\\
		&\qquad\quad\qquad\left.-f_j(\theta)(\cos\alpha, \sin\alpha)\right)+\varepsilon((1+\varepsilon f_i(\theta))(\cos\theta, \sin\theta)-(1+\varepsilon f_j(\alpha))(\cos\alpha, \sin\alpha))^2.
	\end{split}
\end{equation*}

For $k\geq 3$, we will also frequently use the function spaces given in the following, whose norms are  naturally defined as norms of product spaces.
\begin{equation*}
	\begin{split}
		X^k&=\left\{ g\in H^k\, \mid\,\, g(\theta)= \sum\limits_{j=1}^{\infty}a_j\cos(j\theta)+b_j\sin(j\theta)\right\},\\
		X_i^{k}&:=\Bigg\{(f_1, f_2)\in X^{k+1}\times X^k \,\Big|\,\,\begin{cases} -\kappa_i\int\!\!\!\!\!\!\!\!\!\; {}-{} f_1(\theta)\cos(\theta)d\theta= \int\!\!\!\!\!\!\!\!\!\; {}-{} f_2(\theta)\cos(\theta)d\theta,\\ -\kappa_i\int\!\!\!\!\!\!\!\!\!\; {}-{} f_1(\theta)\sin(\theta)d\theta= \int\!\!\!\!\!\!\!\!\!\; {}-{} f_2(\theta)\sin(\theta)d\theta\end{cases}\Bigg\},\\
		Y_i^{k}&:=\Bigg\{(f_1, f_2)\in X^{k}\times X^k \,\Big|\,\,\begin{cases} -\kappa_i\int\!\!\!\!\!\!\!\!\!\; {}-{} f_1(\theta)\sin(\theta)d\theta= \int\!\!\!\!\!\!\!\!\!\; {}-{} f_2(\theta)\cos(\theta)d\theta,\\ \kappa_i\int\!\!\!\!\!\!\!\!\!\; {}-{} f_1(\theta)\cos(\theta)d\theta= \int\!\!\!\!\!\!\!\!\!\; {}-{} f_2(\theta)\sin(\theta)d\theta\end{cases}\Bigg\},\\
		\mathcal X^k&:=\{(\bm f, \bm g)\, \mid\,\, (f_i, g_i)\in X_i^{k},\,\, i=1,\ldots, m\},\\
\,\,\mathcal Y^k&:=\{(\bm f, \bm g)\, \mid\,\, (f_i, g_i)\in Y_i^{k},\,\, i=1,\ldots, m\}.
	\end{split}	
\end{equation*}

  For given $ \bm f=(f_1,\ldots, f_m)$ and $ \bm g=(g_1,\ldots, g_m)$, denote
$\tilde{g}_{i,\varepsilon}(t)=\kappa_i+\varepsilon g_i(t)$. Then we can reduce equations \eqref{2-1} and \eqref{2-3}  to
\begin{equation}\label{2-4}
	\begin{split}
		0&=\mathcal{F}_{i, 1}( \varepsilon,\mathbf{x}, \bm f, \bm g)\\
		&:=\frac{1}{\varepsilon}P.V.\int\!\!\!\!\!\!\!\!\!\; {}-{} \frac{R_i(\alpha)\sin(\theta-\alpha)}{A(\theta, \alpha)+\varepsilon B(f_i,\theta,\alpha)} \tilde{g}_{i,\varepsilon}(\alpha)d\alpha+\frac{1}{R_i(\theta)}P.V.\int\!\!\!\!\!\!\!\!\!\; {}-{} \frac{ f_i'(\theta)R_i(\alpha)(1-\cos(\theta-\alpha))}{A(\theta, \alpha)+\varepsilon B(f_i,\theta,\alpha)} \tilde{g}_{i,\varepsilon}(\alpha)d\alpha\\
		&\quad+\frac{1}{R_i(\theta)}P.V.\int\!\!\!\!\!\!\!\!\!\; {}-{} \frac{\varepsilon f'_i(\theta)(f_i(\theta)-f_i(\alpha))}{A(\theta, \alpha)+\varepsilon B(f_i,\theta,\alpha)} \tilde{g}_{i,\varepsilon}(\alpha)d\alpha\\
		&\quad+\sum_{j\not=i}\frac{1}{R_i(\theta)}\int\!\!\!\!\!\!\!\!\!\; {}-{} \frac{(x_{i}-x_{j})\cdot \left(R_i(\theta)(-\sin\theta, \cos\theta)+\varepsilon f'_i(\theta) (\cos\theta, \sin\theta)\right) }{A_{ij}+\varepsilon B_{ij}(\theta,\alpha)} \tilde{g}_{j,\varepsilon}(\alpha)d\alpha\\
		&\quad+\sum_{j\not=i}\frac{1}{R_i(\theta)}\int\!\!\!\!\!\!\!\!\!\; {}-{} \frac{\varepsilon^2f'_i(\theta)R_i(\theta)-\varepsilon^2f'_i(\theta)R_j(\alpha)\cos(\theta-\alpha)+ \varepsilon R_i(\theta)R_j(\alpha)\sin(\theta-\alpha)}{A_{ij}+\varepsilon B_{ij}(\theta,\alpha)} \tilde{g}_{i,\varepsilon}(\alpha)d\alpha\\
		&\quad-\sum_{j=1}^m  \frac{2\pi}{R_i(\theta)}\int\!\!\!\!\!\!\!\!\!\; {}-{} \nabla H(z_i(\theta),z_j(\alpha))\cdot \left(R_i(\theta)(-\sin\theta, \cos\theta)+\varepsilon f'_i(\theta)(\cos\theta, \sin\theta)\right)\tilde{g}_{i,\varepsilon}(\alpha) d\alpha\\
		 &=:\mathcal{F}_{i,11}+\mathcal{F}_{i,12}+\mathcal{F}_{i,13}+\mathcal{F}_{i,14}
+\mathcal{F}_{i,15}+\mathcal{F}_{i,16},
	\end{split}
\end{equation}
and
\begin{equation}\label{2-5}
0=\mathcal{F}_{i,2}:=(I-P_0)\tilde{\mathcal{F}}_{i,2},
\end{equation}
where $\tilde{\mathcal{F}}_{i,2}$ is given by
\begin{equation}\label{2-6}
	\begin{split}
		&\quad\tilde{\mathcal{F}}_{i,2}( \varepsilon,\mathbf{x}, \bm f, \bm g) :=\frac{\tilde{g}_{i,\varepsilon}(\theta)}{\varepsilon(R_i(\theta)^2+(R_i'(\theta))^2)}P.V.\int\!\!\!\!\!\!\!\!\!\; {}-{} \frac{\varepsilon f_i'(\theta)R_i(\alpha)\sin(\theta-\alpha)}{A(\theta, \alpha)+\varepsilon B(f_i,\theta,\alpha)} \tilde{g}_{i,\varepsilon}(\alpha)d\alpha\\
		 &+\frac{\tilde{g}_{i,\varepsilon}(\theta)}{\varepsilon(R_i(\theta)^2+(R'_i(\theta))^2)}P.V.\int\!\!\!\!\!\!\!\!\!\; {}-{} \frac{R_i(\theta)R_i(\alpha)(\cos(\theta-\alpha)-1)}{A(\theta, \alpha)+\varepsilon B(f_i,\theta,\alpha)} \tilde{g}_{i,\varepsilon}(\alpha)d\alpha\\
		&+\frac{\tilde{g}_{i,\varepsilon}(\theta)}{R_i(\theta)^2+(R_i'(\theta))^2}P.V.\int\!\!\!\!\!\!\!\!\!\; {}-{} \frac{R_i(\theta)( f_i(\alpha)-f_i(\theta))}{A(\theta, \alpha)+\varepsilon B(f_i,\theta,\alpha)} \tilde{g}_{i,\varepsilon}(\alpha)d\alpha\\
		 &+\sum_{j\not=i}\frac{\tilde{g}_{i,\varepsilon}(\theta)}{R_i(\theta)^2+(R_i'(\theta))^2}\int\!\!\!\!\!\!\!\!\!\; {}-{} \frac{(x_{i}-x_{j})^\perp\cdot \left(R_i(\theta)(-\sin\theta, \cos\theta)+\varepsilon f'_i(\theta) (\cos\theta, \sin\theta)\right)}{A_{ij}+\varepsilon B_{ij}(\theta,\alpha)} \tilde{g}_{j,\varepsilon}(\alpha)d\alpha\\
		 &+\sum_{j\not=i}\frac{\tilde{g}_{i,\varepsilon}(\theta)}{R_i(\theta)^2+(R_i'(\theta))^2}\int\!\!\!\!\!\!\!\!\!\; {}-{} \frac{-\varepsilon R_i^2(\theta)+\varepsilon^2f'_i(\theta)R_j(\alpha)\sin(\theta-\alpha)+ \varepsilon R_i(\theta)R_j(\alpha)\cos(\theta-\alpha)}{A_{ij}+\varepsilon B_{ij}(\theta,\alpha)} \tilde{g}_{j,\varepsilon}(\alpha)d\alpha\\
		&-\sum_{j=1}^m  \frac{2\pi\tilde{g}_{i,\varepsilon}(\theta)}{R_i(\theta)^2+(R'_i(\theta))^2}\int\!\!\!\!\!\!\!\!\!\; {}-{} \nabla^\perp H(z_i(\theta),z_j(\alpha))\cdot [R_i(\theta)(-\sin\theta, \cos\theta)+\varepsilon f'_i(\theta)(\cos\theta, \sin\theta)]\\
		&\quad\times\tilde{g}_{i,\varepsilon}(\alpha) d\alpha =:\mathcal{F}_{i,21}+\mathcal{F}_{i,22}+\mathcal{F}_{i,23}+\mathcal{F}_{i,24}+\mathcal{F}_{i,25}+\mathcal{F}_{i,26}.
	\end{split}
\end{equation}

\section{Extension and regularity of functionals}
To apply the implicit function theorem at $\varepsilon=0$, we need to extend the functions $\mathcal{F}_{i,1}$ and $\mathcal{F}_{i, 2}$ defined in Section \ref{2} to $\varepsilon\leq 0$ and check the $C^1$ regularity.

 Let us first show the continuity of these functionals. Letting $V$ be the unit ball in $\left(X^{k+1}\times X^k\right)^m$ and $B_{r_0}(\mathbf{x}_0)$ be the ball centered at $\mathbf{x_0}$ in $\Omega^m$, we have the following proposition.
\begin{proposition}\label{p3-1}
	The functionals  $\mathcal{F}_{i,1}$ and $\mathcal{F}_{i,2}$ can be extended from $(-\varepsilon_0, \varepsilon_0)\times B_{r_0}(\mathbf{x}_0) \times V$ to $X^{k}\times X^k$ as continuous functionals.
\end{proposition}
\begin{proof}
	Throughout the proof, we will frequently use the following Taylor's formula:
	\begin{equation}\label{3-1}
		\frac{1}{(A+B)^\lambda}=\frac{1}{A^\lambda}-\lambda\int_0^1\frac{B}{(A+tB)^{1+\lambda}}dt.
	\end{equation}
	Let us consider $\mathcal{F}_{i,1}$ first. We need to prove that $\partial^l \mathcal{F}_{i,1}\in L^2$ for $l=0,1, \ldots, k$.
	For the first term $$\mathcal{F}_{i, 11}:=\frac{1}{\varepsilon}P.V.\int\!\!\!\!\!\!\!\!\!\; {}-{} \frac{(1+\varepsilon f_i(\alpha))\sin(\theta-\alpha)}{A(\theta, \alpha)+\varepsilon B(f_i,\theta,\alpha)} (\kappa_i+\varepsilon g_i(\alpha))d\alpha,$$  since $R_i(x)=1+\varepsilon f_i(x)$, the possible singularity caused by $\varepsilon=0$ may occur only when we take zeroth order derivative of $\mathcal{F}_{i,11}$. Thus, we first show that $\mathcal{F}_{i, 11}\in L^2$.
	We decompse the kernel into two parts
	\begin{equation}\label{3-2}
		\frac{1}{A(\theta, \alpha)+\varepsilon B(f_i,\theta,\alpha)}
		=\frac{1}{4\sin^2\left(\frac{\theta-\alpha}{2}\right)}\cdot\frac{1}{1+2\varepsilon f_i(\theta)+\varepsilon^2(f_i(\theta)^2+f'_i(\theta)^2)} +\mathcal{K}_R,
	\end{equation}
	where $\mathcal{K}_R:=\frac{1}{A(\theta, \alpha)+\varepsilon B(f_i,\theta,\alpha)}-\frac{1}{4\sin^2\left(\frac{\theta-\alpha}{2}\right)}\cdot\frac{1}{1+2\varepsilon f_i(\theta)+\varepsilon^2(f_i(\theta)^2+f'_i(\theta)^2)}$ is regular than $\frac{1}{4\sin^2\left(\frac{\theta-\alpha}{2}\right)}$. Indeed, by using \eqref{3-1}, we calculate
	\begin{equation*}
		\begin{split}
			&\quad\sin\left(\theta-\alpha\right)\mathcal{K}_R\\
			&=\frac{\sin\left(\theta-\alpha\right)}{A(\theta, \alpha)+\varepsilon B(f_i,\theta,\alpha)}-\frac{\sin\left(\theta-\alpha\right)}{4\sin^2\left(\frac{\theta-\alpha}{2}\right)}\cdot\frac{1}{1+2\varepsilon f_i(\theta)+\varepsilon^2(f_i(\theta)^2+f'_i(\theta)^2)}\\
			&=\frac{\sin\left(\theta-\alpha\right)}{1+2\varepsilon f_i(\theta)+\varepsilon^2(f_i(\theta)^2+f'_i(\theta)^2)}\frac{\left(1+2\varepsilon f_i(\theta)+\varepsilon^2(f_i(\theta)^2+f'_i(\theta)^2)\right)4\sin^2\left(\frac{\theta-\alpha}{2}\right)-A(\theta, \alpha)-\varepsilon B(f_i,\theta,\alpha)}{4\sin^2\left(\frac{\theta-\alpha}{2}\right)(A(\theta, \alpha)+\varepsilon B(f_i,\theta,\alpha))}\\
			&=\frac{\varepsilon \sin\left(\theta-\alpha\right)}{1+2\varepsilon f_i(\theta)+\varepsilon^2(f_i(\theta)^2+f'_i(\theta)^2)}\frac{ f_i(\theta)-f_i(\alpha)+\varepsilon f_i(\theta)(f_i(\theta)-f_i(\alpha))+\varepsilon\frac{4f'_i(\theta)^2\sin^2\left(\frac{\theta-\alpha}{2}\right)-(f_i(\theta)-f_i(\alpha))^2}{4\sin^2\left(\frac{\theta-\alpha}{2}\right)}}{A(\theta, \alpha)+\varepsilon B(f_i,\theta,\alpha)}\\
			 &=\varepsilon\left(\frac{(f_i(\theta)-f_i(\alpha))\sin\left(\theta-\alpha\right)}{4\sin^2\left(\frac{\theta-\alpha}{2}\right)}+O(\varepsilon)\right),
		\end{split}
	\end{equation*}
	where the constant in $O(\varepsilon)$ depends on $\|f\|_{W^{2,\infty}}\leq C\|f\|_{H^3}$.
	This implies  $$ |\sin\left(\theta-\alpha\right)\mathcal{K}_R|\leq C\varepsilon.$$
	Then, it is easy to see that
	\begin{equation*}
		\begin{split}
			&\frac{1}{\varepsilon}\int\!\!\!\!\!\!\!\!\!\; {}-{} \mathcal{K}_R\sin(\theta-\alpha) (1+\varepsilon f_i(\alpha))(\kappa_i+\varepsilon g_i(\alpha))d\alpha\\
			=&\int\!\!\!\!\!\!\!\!\!\; {}-{}\frac{(1+\varepsilon f_i(\alpha))(\kappa_i+\varepsilon g_i(\alpha))(f_i(\theta)-f_i(\alpha))\sin\left(\theta-\alpha\right)}{4\sin^2\left(\frac{\theta-\alpha}{2}\right)}+O(\varepsilon)\\
			=&\int\!\!\!\!\!\!\!\!\!\; {}-{}\frac{\kappa_i(f_i(\theta)-f_i(\alpha))\sin\left(\theta-\alpha\right)}{4\sin^2\left(\frac{\theta-\alpha}{2}\right)}+\varepsilon \mathcal{R}_{111}\\
			=&P.V.\int\!\!\!\!\!\!\!\!\!\; {}-{}\frac{-\kappa_i f_i(\alpha)\sin\left(\theta-\alpha\right)}{4\sin^2\left(\frac{\theta-\alpha}{2}\right)}+\varepsilon \mathcal{R}_{111},
		\end{split}
	\end{equation*}
	where $\mathcal{R}_{111}$ is regular and bounded.
	Hence, to prove $\mathcal{F}_{i, 11}\in L^2$, we  only need to estimate the rest term $\frac{1}{1+2\varepsilon f_i(\theta)+\varepsilon^2(f_i(\theta)^2+f'_i(\theta)^2)}\frac{1}{\varepsilon}P.V.\int\!\!\!\!\!\!\!\!\!\; {}-{} \frac{\sin(\theta-\alpha)}{4\sin^2\left(\frac{\theta-\alpha}{2}\right)} (1+\varepsilon f_i(\alpha))(\kappa_i+\varepsilon g_i(\alpha))d\alpha$. By the odd symmetry and \eqref{3-1}, we have
	\begin{equation*}
		\begin{split}
			&\frac{1}{1+2\varepsilon f_i(\theta)+\varepsilon^2(f_i(\theta)^2+f'_i(\theta)^2)}\frac{1}{\varepsilon}P.V.\int\!\!\!\!\!\!\!\!\!\; {}-{} \frac{\sin(\theta-\alpha)}{4\sin^2\left(\frac{\theta-\alpha}{2}\right)} (1+\varepsilon f_i(\alpha))(\kappa_i+\varepsilon g_i(\alpha))d\alpha\\
			=&\frac{1}{1+2\varepsilon f_i(\theta)+\varepsilon^2(f_i(\theta)^2+f'_i(\theta)^2)} P.V.\int\!\!\!\!\!\!\!\!\!\; {}-{} \frac{\sin(\theta-\alpha)}{4\sin^2\left(\frac{\theta-\alpha}{2}\right)} (\kappa_if_i(\alpha)+ g_i(\alpha)+\varepsilon f_i(\alpha) g_i(\alpha))d\alpha\\
			=&P.V.\int\!\!\!\!\!\!\!\!\!\; {}-{} \frac{(\kappa_if_i(\alpha)+ g_i(\alpha))\sin(\theta-\alpha)}{4\sin^2\left(\frac{\theta-\alpha}{2}\right)} d\alpha+\varepsilon \mathcal{R}_{112}.
		\end{split}
	\end{equation*}
Making the expansion $f_i(\alpha)=f_i(\theta)+O(|\sin\left(\frac{\theta-\alpha}{2}\right)|)$ and $g_i(\alpha)=g_i(\theta)+O(|\sin\left(\frac{\theta-\alpha}{2}\right)|)$, then, we find
	\begin{equation*}
			P.V.\int\!\!\!\!\!\!\!\!\!\; {}-{} \frac{(\kappa_if_i(\alpha)+ g_i(\alpha))\sin(\theta-\alpha)}{4\sin^2\left(\frac{\theta-\alpha}{2}\right)} d\alpha
			=\int\!\!\!\!\!\!\!\!\!\; {}-{} \frac{\sin(\theta-\alpha)}{4\sin^2\left(\frac{\theta-\alpha}{2}\right)}  O(|\sin\left(\frac{\theta-\alpha}{2}\right))d\alpha=O(1),
	\end{equation*}
	where we have used the fact $\left|\frac{\sin(\theta-\alpha)}{4\sin^2\left(\frac{\theta-\alpha}{2}\right)}  O(\sin\left(\frac{\theta-\alpha}{2}\right))\right|\leq C$.
	Therefore, it holds that $P.V.\int\!\!\!\!\!\!\!\!\!\; {}-{} \frac{(\kappa_if_i(\alpha)+ g_i(\alpha))\sin(\theta-\alpha)}{4\sin^2\left(\frac{\theta-\alpha}{2}\right)} d\alpha$ belongs to $L^\infty$ and hence belongs to $L^2$. Moreover, $\mathcal{R}_{112}$ is regular and bounded.
	We conclude  that $\mathcal{F}_{i,11}\in L^\infty$. Furthermore, it holds
	\begin{equation}\label{3-3}
		\begin{split}
			\mathcal{F}_{i,11}=&\frac{1}{1+2\varepsilon f_i(\theta)+\varepsilon^2(f_i(\theta)^2+f'_i(\theta)^2)}\frac{1}{\varepsilon}P.V.\int\!\!\!\!\!\!\!\!\!\; {}-{} \frac{\sin(\theta-\alpha)}{4\sin^2\left(\frac{\theta-\alpha}{2}\right)} (1+\varepsilon f_i(\alpha))(\kappa_i+\varepsilon g_i(\alpha))d\alpha\\
			\qquad\qquad\,+&\frac{1}{\varepsilon}\int\!\!\!\!\!\!\!\!\!\; {}-{} \mathcal{K}_R\sin(\theta-\alpha) (1+\varepsilon f_i(\alpha))(\kappa_i+\varepsilon g_i(\alpha))d\alpha\\
			=&P.V.\int\!\!\!\!\!\!\!\!\!\; {}-{} \frac{ g_i(\alpha))\sin(\theta-\alpha)}{4\sin^2\left(\frac{\theta-\alpha}{2}\right)} d\alpha+\varepsilon \mathcal{R}_{11},
		\end{split}
	\end{equation}
	where $\mathcal{R}_{11}=\mathcal{R}_{111}+\mathcal{R}_{112}$ is regular and bounded.

	Next, we prove that $\partial^k \mathcal{F}_{i,11}\in L^2$. For convince, we rewrite $\mathcal{F}_{i,11}$ as follows  by changing the variable $\alpha$ to $\theta-\alpha$
	$$\mathcal{F}_{i, 11}:=\frac{1}{\varepsilon}P.V.\int\!\!\!\!\!\!\!\!\!\; {}-{} \frac{(1+\varepsilon f_i(\theta-\alpha))\sin(\alpha)}{A(\theta, \theta-\alpha)+\varepsilon B(f_i,\theta,\theta-\alpha)} (\kappa_i+\varepsilon g_i(\theta-\alpha))d\alpha.$$
	Taking $k$th derivatives of $\mathcal{F}_{i,11}$,  we see that the most singular term is
	\begin{equation*}
		\begin{array}{ll}
 P.V.\int\!\!\!\!\!\!\!\!\!\; {}-{} \frac{ \partial^k f_i(\theta-\alpha)\sin(\alpha)}{A(\theta, \theta-\alpha)+\varepsilon B(f_i,\theta,\theta-\alpha)} (\kappa_i+\varepsilon g_i(\theta-\alpha))d\alpha &\\
\,+P.V.\int\!\!\!\!\!\!\!\!\!\; {}-{} \frac{(1+\varepsilon f_i(\theta-\alpha))\sin(\alpha)}{A(\theta, \theta-\alpha)+\varepsilon B(f_i,\theta,\theta-\alpha)}  \partial^k g_i(\theta-\alpha)d\alpha &\\
\,-P.V.\int\!\!\!\!\!\!\!\!\!\; {}-{} \frac{(1+\varepsilon f_i(\theta-\alpha))(\kappa_i+\varepsilon g_i(\theta-\alpha))\sin(\alpha)}{(A(\theta, \theta-\alpha)+\varepsilon B(f_i,\theta,\theta-\alpha))^2}\left[4(\partial^kf_i(\theta)+\partial^k f_i(\theta-\alpha))\sin^2\left(\frac{\theta-\alpha}{2}\right)\right.  &\\
\, \left. +2\varepsilon(f_i(\theta)-f_i(\theta-\alpha))(\partial^kf_i(\theta)-\partial^kf_i(\theta-\alpha))\right. &\\
\,\left.+4\varepsilon (\partial^kf_i(\theta)f_i(\theta-\alpha)+f_i(\theta)\partial^kf_i(\theta-\alpha))\sin^2\left(\frac{\theta-\alpha}{2}\right)\right] d\alpha &\\
=:I_1+I_2+I_3.
		\end{array}
	\end{equation*}
	We first deal with the term $I_1$. By the splitting of the kernel \eqref{3-2}, we derive
	\begin{equation*}
		\begin{split}
			&	I_1= P.V.\int\!\!\!\!\!\!\!\!\!\; {}-{} \frac{ \partial^k f_i(\theta-\alpha)\sin(\alpha)}{A(\theta, \theta-\alpha)+\varepsilon B(f_i,\theta,\theta-\alpha)} (\kappa_i+\varepsilon g_i(\theta-\alpha))d\alpha\\
			&\,\quad =P.V.\int\!\!\!\!\!\!\!\!\!\; {}-{} \frac{ \partial^k f_i(\alpha)\sin(\theta-\alpha)}{A(\theta, \alpha)+\varepsilon B(f_i,\theta,\alpha)} (\kappa_i+\varepsilon g_i(\alpha))d\alpha\\
			&\,\quad =\frac{1}{1+2\varepsilon f_i(\theta)+\varepsilon^2(f_i(\theta)^2+f'_i(\theta)^2)} P.V.\int\!\!\!\!\!\!\!\!\!\; {}-{}\frac{\partial^k f_i(\alpha)(\kappa_i+\varepsilon g_i(\alpha))\sin(\theta-\alpha)}{4\sin^2\left(\frac{\theta-\alpha}{2}\right)} \\
			&\qquad+P.V.\int\!\!\!\!\!\!\!\!\!\; {}-{}\mathcal{K}_R \sin(\theta-\alpha)\partial^k f_i(\alpha)(\kappa_i+\varepsilon g_i(\alpha))d\alpha.
		\end{split}
	\end{equation*}
	Noting that $|\mathcal{K}_R \sin(\theta-\alpha)|\leq C\varepsilon $, we have $\|P.V.\int\!\!\!\!\!\!\!\!\!\; {}-{}\mathcal{K}_R \sin(\theta-\alpha)\partial^k f_i(\alpha)(\kappa_i+\varepsilon g_i(\alpha))d\alpha\|_{L^2}\leq C||f_i||_{H^k}\|g\|_{L^2}$ is bounded. Since $P.V.\int\!\!\!\!\!\!\!\!\!\; {}-{}\frac{\partial^k f_i(\alpha)(\kappa_i+\varepsilon g_i(\alpha))\sin(\theta-\alpha)}{4\sin^2\left(\frac{\theta-\alpha}{2}\right)}$ is the Hilbert transformation of $\partial^k f_i(\alpha)(\kappa_i+\varepsilon g_i(\alpha))$ and hence $$\left\|P.V.\int\!\!\!\!\!\!\!\!\!\; {}-{}\frac{\partial^k f_i(\alpha)(\kappa_i+\varepsilon g_i(\alpha))\sin(\theta-\alpha)}{4\sin^2\left(\frac{\theta-\alpha}{2}\right)}\right\|_{L^2}\leq \|\partial^k f_i(\alpha)(\kappa_i+\varepsilon g_i(\alpha))\|_{L^2}\leq ||f||_{H^k}||g||_{L^2}.$$
	
	Similarly, one can check that $||I_2||_{L^2}\leq ||f||_{L^2}||g||_{H^k}$.
	
	To estimate the last term $I_3$, we split the kernel as follows
\begin{equation*}
			\frac{4\sin^2\left(\frac{\alpha}{2}\right)}{(A(\theta, \theta-\alpha)+\varepsilon B(f_i,\theta,\theta-\alpha))^2}
			=\frac{1}{4\sin^2\left(\frac{\alpha}{2}\right)}\cdot\frac{1}{(1+2\varepsilon f_i(\theta)+\varepsilon^2(f_i(\theta)^2+f'_i(\theta)^2))^2} +\tilde{\mathcal{K}}_R,
\end{equation*}
	where $\tilde{\mathcal{K}}_R$  satisfies $| \tilde{\mathcal{K}}_R\sin\alpha|\leq C$. Since convolution with the kernel  $\frac{\sin\alpha}{4\sin^2\left(\frac{\alpha}{2}\right)}$ defines the Hilbert transformation, we find that \begin{tiny} $$P.V.\int\!\!\!\!\!\!\!\!\!\; {}-{} \frac{(1+\varepsilon f_i(\theta-\alpha))(\kappa_i+\varepsilon g_i(\theta-\alpha))\sin(\alpha)}{(A(\theta, \theta-\alpha)+\varepsilon B(f_i,\theta,\theta-\alpha))^2}\sin^2\left(\frac{\theta-\alpha}{2}\right)\left((\partial^kf_i(\theta)+\partial^k f_i(\theta-\alpha))
		(\partial^kf_i(\theta)f_i(\theta-\alpha)+f_i(\theta)\partial^kf_i(\theta-\alpha))\right) d\alpha$$ \end{tiny} belongs to $L^2$ due to the $L^2$ boundedness of Hilbert transformation and the regularity of $\tilde{\mathcal{K}}_R$.
	For the remaining term in $I_3$
	$$2\varepsilon P.V.\int\!\!\!\!\!\!\!\!\!\; {}-{} \frac{(1+\varepsilon f_i(\theta-\alpha))(\kappa_i+\varepsilon g_i(\theta-\alpha))\sin(\alpha)}{(A(\theta, \theta-\alpha)+\varepsilon B(f_i,\theta,\theta-\alpha))^2}(f_i(\theta)-f_i(\theta-\alpha))(\partial^kf_i(\theta)-\partial^kf_i(\theta-\alpha)) d\alpha,$$ we decompose the kernel
	\begin{equation*}
			\frac{(1+\varepsilon f_i(\theta-\alpha))(\kappa_i+\varepsilon g_i(\theta-\alpha))\sin(\alpha)}{(A(\theta, \theta-\alpha)+\varepsilon B(f_i,\theta,\theta-\alpha))^2}(f_i(\theta)-f_i(\theta-\alpha))\\
			=\frac{(1+\varepsilon f_i(\theta))(\kappa_i+\varepsilon g_i(\theta))f'_i(\theta)}{4\sin^2\left(\frac{\alpha}{2}\right)}+\bar{\mathcal{K}}_R,
	\end{equation*}
	where $\bar{\mathcal{K}}_R$ satisfies $|\bar{\mathcal{K}}_R\sin\alpha|\leq C$.	
	Then, we deduce
	\begin{equation*}
		\begin{split}
			&	P.V.\int\!\!\!\!\!\!\!\!\!\; {}-{} \frac{(1+\varepsilon f_i(\theta-\alpha))(\kappa_i+\varepsilon g_i(\theta-\alpha))\sin(\alpha)}{(A(\theta, \theta-\alpha)+\varepsilon B(f_i,\theta,\theta-\alpha))^2}(f_i(\theta)-f_i(\theta-\alpha))(\partial^kf_i(\theta)-\partial^kf_i(\theta-\alpha)) d\alpha\\
			&=(1+\varepsilon f_i(\theta))(\kappa_i+\varepsilon g_i(\theta))f'(\theta)P.V.\int\!\!\!\!\!\!\!\!\!\; {}-{} \frac{\partial^kf_i(\theta)-\partial^kf_i(\theta-\alpha)}{4\sin^2\left(\frac{\alpha}{2}\right)} d\alpha+P.V.\int\!\!\!\!\!\!\!\!\!\; {}-{} \bar{\mathcal{K}}_R(\partial^kf_i(\theta)-\partial^kf_i(\theta-\alpha)) d\alpha\\
			&=(1+\varepsilon f_i(\theta))(\kappa_i+\varepsilon g_i(\theta))f'_i(\theta)\left((-\Delta)^{\frac{1}{2}}(\partial^kf_i)\right)(\theta)+P.V.\int\!\!\!\!\!\!\!\!\!\; {}-{} \bar{\mathcal{K}}_R(\partial^kf_i(\theta)-\partial^kf_i(\theta-\alpha)) d\alpha.
		\end{split}
	\end{equation*}
	By Fourier transformation and Hardy inequality, we obtain
	$$||(-\Delta)^{\frac{1}{2}}(\partial^kf_i)||_{L^2}\leq ||\nabla \partial^kf_i||_{L^2}\leq ||f_i||_{H^{k+1}}$$ and $$\left\|\int\!\!\!\!\!\!\!\!\!\; {}-{} \bar{\mathcal{K}}_R(\partial^kf_i(\theta)-\partial^kf_i(\theta-\alpha)) d\alpha\right\|_{L^2}\leq  \|\nabla \partial^kf_i||_{L^2}\leq ||f_i||_{H^{k+1}}.$$
	Consequently, we have $\partial^k \mathcal{F}_{i,11}\in L^2$ and hence $\mathcal{F}_{i,11}\in H^k$.

	Now, we turn to the second term
	\begin{equation*}
		\mathcal{F}_{i,12}:=\frac{1}{1+\varepsilon f_i(\theta)}P.V.\int\!\!\!\!\!\!\!\!\!\; {}-{} \frac{ f_i'(\theta)(1+\varepsilon f_i(\alpha))(1-\cos(\theta-\alpha))}{A(\theta, \alpha)+\varepsilon B(f_i,\theta,\alpha)} (\kappa_i+\varepsilon g_i(\alpha))d\alpha.
	\end{equation*}
	Since $|1-\cos(\theta-\alpha)|=\sin^2\left(\frac{\theta-\alpha}{2}\right)$, the kernel of this term is actually regular and bounded. Therefore, it is easy to see that $\mathcal{F}_{i,12}\in H^k$. Moreover, by \eqref{3-1}, we find
	\begin{equation}\label{3-4}
		\begin{split}
			\mathcal{F}_{12}&=\frac{1}{1+\varepsilon f_i(\theta)}P.V.\int\!\!\!\!\!\!\!\!\!\; {}-{} \frac{ f'_i(\theta)(1+\varepsilon f_i(\alpha))(1-\cos(\theta-\alpha))}{A(\theta, \alpha)+\varepsilon B(f_i,\theta,\alpha)} (\kappa_i+\varepsilon g_i(\alpha))d\alpha\\
			&=P.V.\int\!\!\!\!\!\!\!\!\!\; {}-{} \frac{ \kappa_if'_i(\theta)(1-\cos(\theta-\alpha))}{A(\theta, \alpha)}d\alpha+\varepsilon \mathcal{R}_{12}\\
			&=\frac{\kappa_i}{2}f_i'(\theta)+ \varepsilon \mathcal{R}_{12},
		\end{split}
	\end{equation}
	where $\mathcal{R}_{12}$ is smooth and we have used the identity $1-\cos(\theta-\alpha)=2\sin^2\left(\frac{\theta-\alpha}{2}\right)=\frac{A(\theta, \alpha)}{2}$.
	
	For  $\mathcal{F}_{i,13}$, taking $k$th derivatives of $\mathcal{F}_{i,13}$,  we see that the most singular terms are
	\begin{equation*}
		\begin{array}{ll}
			\frac{\varepsilon\partial ^{k+1}f_i(\theta)}{1+\varepsilon f_i(\theta)}P.V.\int\!\!\!\!\!\!\!\!\!\; {}-{} \frac{ f_i(\theta)-f_i(\alpha)}{A(\theta, \alpha)+\varepsilon B(f_i,\theta,\alpha)} (\kappa_i+\varepsilon g_i(\alpha))d\alpha &\\
\,+\frac{\varepsilon f'_i(\theta)}{1+\varepsilon f_i(\theta)}P.V.\int\!\!\!\!\!\!\!\!\!\; {}-{} \frac{\partial ^{k}f_i(\theta)-\partial ^{k}f_i(\theta-\alpha)}{A(\theta, \theta-\alpha)+\varepsilon B(f_i,\theta,\theta-\alpha)} (\kappa_i+\varepsilon g_i(\theta-\alpha))d\alpha &\\
\,+\frac{\varepsilon^2 f_i'(\theta)}{1+\varepsilon f_i(\theta)}P.V.\int\!\!\!\!\!\!\!\!\!\; {}-{} \frac{(f_i(\theta)-f_i(\theta-\alpha))\partial ^{k}g_i(\theta-\alpha)}{A(\theta, \theta-\alpha)+\varepsilon B(f_i,\theta,\theta-\alpha)} d\alpha &\\
\,-\frac{\varepsilon f_i'(\theta)}{1+\varepsilon f_i(\theta)}P.V.\int\!\!\!\!\!\!\!\!\!\; {}-{} \frac{(f_i(\theta)-f_i(\theta-\alpha))(\kappa_i+\varepsilon g_i(\theta-\alpha))}{(A(\theta, \theta-\alpha)+\varepsilon B(f_i,\theta,\theta-\alpha))^2} \left[4(\partial^kf_i(\theta)+\partial^k f_i(\theta-\alpha))\sin^2\left(\frac{\theta-\alpha}{2}\right)\right.&\\ \qquad\qquad\qquad\qquad\,\left.+2\varepsilon(f_i(\theta)-f_i(\theta-\alpha))(\partial^kf_i(\theta)
-\partial^kf_i(\theta-\alpha))\right.&\\
\qquad\qquad\qquad\qquad\,\left.+4\varepsilon(\partial^kf_i(\theta)f_i(\theta-\alpha)
+f_i(\theta)\partial^kf_i(\theta-\alpha))\sin^2\left(\frac{\theta-\alpha}{2}\right)\right]d\alpha &\\
 =:J_1+J_2+J_3+J_4. &
\end{array}
	\end{equation*}
	Since $P.V.\int\!\!\!\!\!\!\!\!\!\; {}-{} \frac{ f_i(\theta)-f_i(\alpha)}{A(\theta, \alpha)+\varepsilon B(f_i,\theta,\alpha)} (\kappa_i+\varepsilon g_i(\alpha))d\alpha\in L^2$, by Taylor's expansion of $f_i(\alpha)$, we know that the first term $J_1$ is bounded in $L^2$.
	To deal with $J_2$, we split the kernel
	\begin{equation*}
		\begin{split}
			\frac{\kappa_i+\varepsilon g_i(\theta-\alpha)}{A(\theta, \theta-\alpha)+\varepsilon B(f_i,\theta,\theta-\alpha)}
			=\frac{\kappa_i+\varepsilon g_i(\theta)}{4\sin^2\left(\frac{\alpha}{2}\right)}+\hat{\mathcal{K}}_R,
		\end{split}
	\end{equation*}
	where $|\hat{\mathcal{K}}_R\sin\frac{\alpha}{2}|\leq C$. Therefore, we conclude
	\begin{equation*}
		\begin{split}
			&\frac{\varepsilon f'_i(\theta)}{1+\varepsilon f_i(\theta)}P.V.\int\!\!\!\!\!\!\!\!\!\; {}-{} \frac{\partial ^{k}f_i(\theta)-\partial ^{k}f_i(\theta-\alpha)}{A(\theta, \theta-\alpha)+\varepsilon B(f_i,\theta,\theta-\alpha)} (\kappa_i+\varepsilon g_i(\theta-\alpha))d\alpha\\
			=&\frac{\varepsilon f'_i(\theta)(\kappa_i+\varepsilon g_i(\theta))}{1+\varepsilon f_i(\theta)}P.V.\int\!\!\!\!\!\!\!\!\!\; {}-{}\frac{\partial ^{k}f_i(\theta)-\partial ^{k}f_i(\theta-\alpha)}{4\sin^2\left(\frac{\alpha}{2}\right)}+ \frac{\varepsilon f'_i(\theta)}{1+\varepsilon f_i(\theta)} P.V.\int\!\!\!\!\!\!\!\!\!\; {}-{} \hat{\mathcal{K}}_R (\partial ^{k}f_i(\theta)-\partial ^{k}f_i(\theta-\alpha)) d\alpha\\
			=&\frac{\varepsilon f_i'(\theta)(\kappa_i+\varepsilon g_i(\theta))}{1+\varepsilon f_i(\theta)} (-\Delta)^{\frac{1}{2}} (\partial ^{k}f_i)(\theta)+ \frac{\varepsilon f'_i(\theta)}{1+\varepsilon f(\theta)} P.V.\int\!\!\!\!\!\!\!\!\!\; {}-{} \hat{\mathcal{K}}_R (\partial ^{k}f_i(\theta)-\partial ^{k}f_i(\theta-\alpha)) d\alpha.
		\end{split}
	\end{equation*}
	By Fourier transformation and Hardy inequality we obtain
	$$||(-\Delta)^{\frac{1}{2}}(\partial^kf_i)||_{L^2}\leq ||\nabla \partial^kf_i||_{L^2}\leq ||f_i||_{H^{k+1}}$$ and $$\left\|\int\!\!\!\!\!\!\!\!\!\; {}-{} \hat{\mathcal{K}}_R(\partial^kf_i(\theta)-\partial^kf_i(\theta-\alpha)) d\alpha\right\|_{L^2}\leq  \|\nabla \partial^kf_i||_{L^2}\leq ||f_i||_{H^{k+1}}.$$
	We can show that the remaining terms $J_3$ and $J_4$ are  bounded in $L^2$ similarly. Moreover, it can be seen that
	\begin{equation}\label{3-5}
		\mathcal{F}_{i,13}=\frac{1}{1+\varepsilon f_i(\theta)}P.V.\int\!\!\!\!\!\!\!\!\!\; {}-{} \frac{\varepsilon f'_i(\theta)(f_i(\theta)-f_i(\alpha))}{A(\theta, \alpha)+\varepsilon B(f_i,\theta,\alpha)} (\kappa_i+\varepsilon g_i(\alpha))d\alpha=\varepsilon \mathcal{R}_{13},
	\end{equation}
	where $\mathcal{R}_{13}$ is regular.
	
	Since $H(x,y)$ is smooth in $\Omega$, the terms $\mathcal{F}_{i,14}$, $\mathcal{F}_{i,15}$ and $\mathcal{F}_{i,16}$ are apparently smooth and belong to $H^k$. Furthermore, we have
	\begin{equation}\label{3-6}
		\begin{array}{ll}
\mathcal{F}_{i,14}+\mathcal{F}_{i,15}+\mathcal{F}_{i,16}\,&\\
=\sum_{j\not=i}\int\!\!\!\!\!\!\!\!\!\; {}-{} \frac{\kappa_j(x_{i}-x_{j})\cdot (-\sin\theta, \cos\theta) }{|x_{i,\varepsilon}-x_{j,\varepsilon}|^2} d\alpha -\sum_{j=1}^m  2\pi\int\!\!\!\!\!\!\!\!\!\; {}-{} \kappa_j\nabla H(x_{i},x_{j})\cdot (-\sin\theta, \cos\theta) d\alpha+\varepsilon \mathcal{R}_{14}&\\
			=\sum_{j\not=i}2\pi\int\!\!\!\!\!\!\!\!\!\; {}-{} \kappa_j\nabla G(x_{i},x_{j})\cdot (-\sin\theta, \cos\theta) d\alpha\,&\\
\qquad- 2\pi\int\!\!\!\!\!\!\!\!\!\; {}-{} \kappa_i\nabla H(x_{i},x_{i})\cdot (-\sin\theta, \cos\theta) d\alpha+\varepsilon \mathcal{R}_{14},&
		\end{array}
	\end{equation}
	where $\mathcal{R}_{14}$ is bounded and smooth.
	
	By \eqref{3-3}, \eqref{3-4}, \eqref{3-5} and \eqref{3-6}, we conclude
	\begin{equation}\label{3-7}
		\begin{array}{ll}
\mathcal{F}_{i,1}( \varepsilon,\mathbf{x}, \bm f, \bm g)
			=P.V.\int\!\!\!\!\!\!\!\!\!\; {}-{} \frac{\sin(\theta-\alpha) g_i(\alpha)}{4\sin^2\left(\frac{\theta-\alpha}{2}\right)} d\alpha+\frac{\kappa_i}{2}f_i'(\theta)
			+\sum_{j\not=i}2\pi \kappa_j\nabla G(x_{i},x_{j})\cdot (-\sin\theta, \cos\theta)\,&\\
\qquad\qquad\qquad\qquad\qquad\,- 2\pi  \kappa_i\nabla H(x_{i},x_{i})\cdot (-\sin\theta, \cos\theta)+\varepsilon \mathcal{R}_{1},&\\
		\end{array}
	\end{equation}
	where $\mathcal{R}_{1}:=\mathcal{R}_{11}+\mathcal{R}_{12}+\mathcal{R}_{13}+\mathcal{R}_{14}$ is regular.
	Hence, we can define
	\begin{equation}\label{3-8}
		\begin{array}{ll}
\mathcal{F}_{i,1}( 0,\mathbf{x}, \bm f, \bm g)
			:=P.V.\int\!\!\!\!\!\!\!\!\!\; {}-{} \frac{\sin(\theta-\alpha) g_i(\alpha)}{4\sin^2\left(\frac{\theta-\alpha}{2}\right)} d\alpha+\frac{\kappa_i}{2}f'_i(\theta)
+\sum_{j\not=i}2\pi \kappa_j\nabla G(x_{i},x_{j})\cdot (-\sin\theta, \cos\theta) \,&\\
\qquad\qquad\qquad\qquad\,- 2\pi \kappa_i\nabla H(x_{i},x_{i})\cdot (-\sin\theta, \cos\theta).&
		\end{array}
	\end{equation}
	
	Next, we prove the continuity of $\mathcal{F}_{i,1}$. By \eqref{3-7} and the definition of $\mathcal{F}_{i,1}( 0,\mathbf{x}, \bm f,\bm g)$, one can easily check that $\mathcal{F}_{i,1}$ is continuous with respect to $\varepsilon$ at $\varepsilon=0$. Thus, we only need to prove that $\mathcal{F}_{i,1}$ is continuous with respect to $\varepsilon$ for $\varepsilon\not=0$. However, it is easy to see that the continuity of $ \mathcal{F}_{i,1}$ with respect to $\varepsilon$ is a consequence of its continuity with respect to $\bm f$ and $\bm g$ when $\varepsilon\not=0$, on which we will focus below.

	We only prove the continuity of $\mathcal{F}_{i,11}$ with respect to $f_i$ and $g_i$, the continuity of other terms in $\mathcal{F}_{i, 1}$ can be proven though similar or even easier method. We will use following notations: for a general function $h$, we denote $\Delta h=h(\theta)-h(\alpha), \ \ \ h=h(\theta), \ \ \ \tilde h=h(\alpha),$ and
	$$D(h)=\varepsilon^2(\Delta h)^2+4(1+\varepsilon h)(1+\varepsilon\tilde h)\sin^2\left(\frac{\theta-\alpha}{2}\right).$$
	To show the continuity of $\mathcal{F}_{i,11}$ with respect to $f_i$, let $(f_1,g),(f_2,g)\in X_i^k$. Then we can calculate the difference
	\begin{equation*}
		\begin{array}{ll}
\mathcal{F}_{11}(\varepsilon, f_2,g)-\mathcal{F}_{11}(\varepsilon, f_1,g)
			=P.V.\int\!\!\!\!\!\!\!\!\!\; {}-{} \frac{ (f_2(\alpha)-f_1(\alpha))\sin(\theta-\alpha)}{D(f_2)} (\kappa+\varepsilon g(\alpha))d\alpha\,&\\
\quad\,+\frac{1}{\varepsilon}P.V.\int\!\!\!\!\!\!\!\!\!\; {}-{} (1+\varepsilon f_1(\alpha))(\kappa+\varepsilon g(\alpha))\sin(\theta-\alpha)\left(\frac{1}{D(f_2)}-\frac{1}{D(f_1)}\right) d\alpha
=:K_1+K_2.&
		\end{array}
	\end{equation*}
	For the first term $K_1$, it is easy to prove $\|K_1\|_{H^k}\leq C\|f_1-f_2\|_{H^{k+1}}$ by the technique we have used before. For the second term $K_2$, since
	\begin{equation*}
		\begin{split}
			&\frac{1}{D(f_2)}-\frac{1}{D(f_1)}\\
			=&\frac{\varepsilon^2((\Delta f_1)^2-(\Delta f_2)^2)+4\varepsilon\big((f_1-f_2)(1+\varepsilon\tilde f_2)+(\tilde f_1-\tilde f_2)(1+\varepsilon f_1)\big)\sin^2(\frac{x-y}{2})}{D(f_1)D(f_2)}\\
			=&\varepsilon\frac{\varepsilon(\Delta f_1+\Delta f_2)(\Delta f_1-\Delta f_2)+4\big((f_1-f_2)(1+\varepsilon\tilde f_2)+(\tilde f_1-\tilde f_2)(1+\varepsilon f_1)\big)\sin^2(\frac{x-y}{2})}{D(f_1)D(f_2)}
		\end{split}
	\end{equation*}
	it holds that the singularity of $\frac{1}{D(f_2)}-\frac{1}{D(f_1)}$ is also of the order $O\left(\frac{1}{4\sin^2\left(\frac{\theta-\alpha}{2}\right)} \right)$, the same as the kernel in $\mathcal{F}_{i,11}$ itself. Therefore, using similar argument as above, we can prove that $\|K_2\|_{H^k}\leq \|f_2-f_1\|_{H^{k+1}}$, which shows the continuity of $\mathcal{F}_{i,11}$ with respect to $f_i$. Notice that  $\mathcal{F}_{i,11}$ is linear with respect to $g_i$. Then the continuity of $\mathcal{F}_{i,11}$ with respect to $g$ can be obtained by similar argument as the proof of boundedness of $\mathcal{F}_{i,11}$ in $H^k$.
	
	We have shown that the conclusion of Proposition \ref{p3-1} holds true for $\mathcal{F}_{i,1}$. 	The fact that $\mathcal{F}_{i,2}$ is well-defined and continuous can be verified in a similar way. Attention should be paid to that projection operator $I- P_0$ eliminates all constant terms in $\tilde{\mathcal{F}}_{i,2}$, which also removes  singularity in $\mathcal{F}_{i,2}$. By \eqref{3-1}, we obtain
	\begin{equation}\label{3-9}
		\begin{array}{ll}
\tilde{\mathcal{F}}_{i, 2}( \varepsilon,\mathbf{x}_\varepsilon, \bm f, \bm g)
			=\kappa_i^2f_i(\theta)-\kappa_i^2P.V.\int\!\!\!\!\!\!\!\!\!\; {}-{} \frac{f_i(\theta)-f_i(\alpha) }{4\sin^2\left(\frac{\theta-\alpha}{2}\right)} d\alpha-\kappa_i\frac{g_i(\theta)}{2}\,&\\
\qquad\qquad\qquad\qquad\,-\sum_{j\not=i}2\pi \kappa_i\kappa_j \nabla^\perp G(x_{i},x_{j})\cdot (-\sin\theta, \cos\theta)\,&\\
\qquad\qquad\qquad\qquad\,+ 2\pi \kappa_i^2\nabla^\perp H(x_{i},x_{i})\cdot (-\sin\theta, \cos\theta)+\varepsilon \mathcal{R}_2,&
		\end{array}
	\end{equation}
	where $\mathcal{R}_{2}$ is smooth. Thus, we define
	\begin{equation}\label{3-10}
		\begin{array}{ll}
\tilde{\mathcal{F}}_{i,2}( 0,\mathbf{x}, \bm f, \bm g)(\theta)
			=\kappa_i^2f_i(\theta)-\kappa_i^2P.V.\int\!\!\!\!\!\!\!\!\!\; {}-{} \frac{f_i(\theta)-f_i(\alpha) }{4\sin^2\left(\frac{\theta-\alpha}{2}\right)} d\alpha-\kappa_i\frac{g_i(\theta)}{2}&\\
\qquad\qquad\qquad\qquad\,\,-\sum_{j\not=i}2\pi \kappa_i\kappa_j \nabla^\perp G(x_{i},x_{j})\cdot (-\sin\theta, \cos\theta)\,&\\
\qquad\qquad\qquad\qquad\,\,+ 2\pi \kappa_i^2\nabla^\perp H(x_{i},x_{i})\cdot (-\sin\theta, \cos\theta).&
		\end{array}
	\end{equation}
\end{proof}

Our next proposition concerns the $C^1$ regularity.
\begin{proposition}\label{p3-2}
	The Gateaux derivatives $\partial_{(\bm f, \bm g)} \mathcal{F}_{i, 1}$ and  $\partial_{(\bm f, \bm g)} \mathcal{F}_{i, 2}$ exist and are continuous.
\end{proposition}
\begin{proof}
	We first prove that the derivative of $\mathcal{F}_{i,11}$ with respect to $f_i$ exists and is as follows,
	\begin{equation}\label{3-11}
		\partial_{f_i} \mathcal{F}_{i,11} h=F_{i,11i}h, \quad \forall \,h\in X^{k+1}
	\end{equation}
	where $F_{i,11i}$ is given by
	\begin{equation}\label{3-12}
		\begin{split}
			F_{i,11i}h
			:=&P.V.\int\!\!\!\!\!\!\!\!\!\; {}-{} \frac{h(\alpha)\sin(\theta-\alpha)}{A(\theta, \alpha)+\varepsilon B(f_i,\theta,\alpha)} (\kappa_i+\varepsilon g_i(\alpha))d\alpha\\
			-&P.V.\int\!\!\!\!\!\!\!\!\!\; {}-{} \frac{(1+\varepsilon f_i(\alpha))(\kappa_i+\varepsilon g_i(\alpha))\sin(\theta-\alpha)}{(A(\theta, \alpha)+\varepsilon B(f_i,\theta,\alpha))^2}  \left[4(h(\theta)+h(\alpha))\sin^2\left(\frac{\theta-\alpha}{2}\right)\right.\\
			&\quad \left.+2\varepsilon (f_i(\theta)-f_i(\alpha))(h(\theta)-h(\alpha))
+4\varepsilon\left(h(\theta)f_i(\alpha)+h(\alpha)f_i(\theta)\right)\sin^2
\left(\frac{\theta-\alpha}{2}\right)\right]d\alpha.
		\end{split}
	\end{equation}
	To prove \eqref{3-1}, one need to verify
	\begin{equation}\label{3-13}
		\lim\limits_{t\to0}\left\|\frac{\mathcal{F}_{i,11}(\varepsilon, f_i+th,g_i)-\mathcal{F}_{i,11}(\varepsilon, f_i,g_i)}{t}-F_{i,11i} h\right\|_{H^{k}}= 0.
	\end{equation}

	Using the notations given in Proposition \ref{p3-1}, we deduce
	\begin{equation*}
		\begin{split}
			&\frac{\mathcal{F}_{i,11}(\varepsilon, f_i+th,g_i)-\mathcal{F}_{i,11}(\varepsilon, f_i,g_i)}{t}-F_{i,11i} h\\
			=&\frac{1}{t\varepsilon} \int\!\!\!\!\!\!\!\!\!\; {}-{}(1+\varepsilon f_i(\alpha))(\kappa_i+\varepsilon g_i(\alpha))\sin(\theta-\alpha)\\
			&\times \left(\frac{1}{D(f_i+th)}-\frac{1}{D(f_i)}+t\frac{2\varepsilon^2\Delta f_i\Delta h+4\varepsilon((1+\varepsilon \tilde f_i)h+(1+\varepsilon f_i(\alpha))\tilde h)\sin^2(\frac{x-y}{2})}{D(f_i)^2}\right)\\
			+&\int\!\!\!\!\!\!\!\!\!\; {}-{} h(\alpha)(\kappa_i+\varepsilon g_i(\alpha))\sin(\theta-\alpha)\left(\frac{1}{D(f_1+th)}-\frac{1}{D(f_1)}\right)dy\\
			=&:F_{111}+F_{112}.
		\end{split}
	\end{equation*}
	By the  mean value theorem, we find
	\begin{equation*}
		\begin{split}
			&\frac{1}{D(f+th_1)}-\frac{1}{D(f)}=O\left( \frac{t\varepsilon}{4\sin^2\left(\frac{x-y}{2}\right)}\right),\\
			&\frac{1}{D(f+th_1)}-\frac{1}{D(f)}+t\frac{2\varepsilon^2\Delta f\Delta h_1+4\big(\varepsilon\tilde R h_1+\varepsilon \tilde h_1R\big)\sin^2(\frac{x-y}{2})}{D(f)^2}=O\left( \frac{t^2\varepsilon^2}{4\sin^2\left(\frac{x-y}{2}\right)}\right),	
		\end{split}
	\end{equation*}
	which means the kernels in $F_{111}$ and $F_{112}$ are of the same order as the kernel in  $\mathcal{F}_{i,11}$. Therefore, by similar argument as Proposition \ref{p3-1}, we have
	\begin{equation*}
		\|F_{111}\|_{H^{k}}+\|F_{112}\|_{H^{k}}\le Ct\|h\|_{X^{k+1}}.
	\end{equation*}
	
Letting $t\rightarrow 0$, we obtain \eqref{3-13} and hence obtain the existence of Gateaux derivative of $\mathcal{F}_{i,11}$. To prove the continuity of $\partial_{f_i} \mathcal{F}_{i,11}(\varepsilon, f_i, g_i)h$, one just need to verify by definition. Since there is no other new idea than the proof of continuity for $\mathcal{F}_{i,11}(\varepsilon, f_i, g_i)$, we omit it. The existence and continuity of Gateaux derivatives of other terms in $\mathcal{F}_{i,1}$ and $\mathcal{F}_{i,2}$ can be obtained via similar argument, which we leave out here.  Noting that $\mathcal{F}_{i,1}$ and $\mathcal{F}_{i,2}$ are almost linear dependent on $g$, it is much easier to compute their Gateaux derivatives with respect to $g$, so we leave them to our reader. For readers' convenience, we also write down the  derivatives of $\mathcal{F}_{i,1}$ and $\mathcal{F}_{i,2}$ in the following form directly without proof here.
	
Recall the definitions $\tilde{g}_{i,\varepsilon}(t)=\kappa_i+\varepsilon g_i(t)$, $R_i(t)=1+\varepsilon f_i(t)$. For any $h_1\in X^{k+1}$ and $h_2\in X^k$, then we have
	\begin{equation}\label{3-14}
		\begin{array}{ll}
\partial_{f_i}\mathcal{F}_{i,1}(\varepsilon,\mathbf{x}, \bm f, \bm g) h_1
			=P.V.\int\!\!\!\!\!\!\!\!\!\; {}-{} \frac{h_1(\alpha)\sin(\theta-\alpha)}{A(\theta, \alpha)+\varepsilon B(f_i,\theta,\alpha)} \tilde{g}_{i,\varepsilon}(\alpha)d\alpha&\\
\qquad\qquad\,-P.V.\int\!\!\!\!\!\!\!\!\!\; {}-{} \frac{R_i(\alpha)\sin(\theta-\alpha)}{(A(\theta, \alpha)+\varepsilon B(f_i,\theta,\alpha))^2} \tilde{g}_{i,\varepsilon}(\alpha) \left[4(h_1(\theta)+h_1(\alpha))\sin^2\left(\frac{\theta-\alpha}{2}\right)\right.&\\
\qquad\qquad\quad\, \left.+2\varepsilon (f_i(\theta)-f_i(\alpha))(h_1(\theta)-h_1(\alpha))+4\varepsilon\left(h_1(\theta)f_i(\alpha)+h_1(\alpha)f_i(\theta)\right)\sin^2\left(\frac{\theta-\alpha}{2}\right)\right]d\alpha\\
\qquad\qquad-\frac{\varepsilon h_1(\theta)}{R_i(\theta)^2}P.V.\int\!\!\!\!\!\!\!\!\!\; {}-{} \frac{ f_i'(\theta)\tilde{f}_{i,\varepsilon}(\alpha)(1-\cos(\theta-\alpha))}{A(\theta, \alpha)+\varepsilon B(f_i,\theta,\alpha)} \tilde{g}_{i,\varepsilon}(\alpha)d\alpha &\\
\qquad\qquad\,+\frac{1}{R_i(\theta)}P.V.\int\!\!\!\!\!\!\!\!\!\; {}-{} \frac{ (h_1'(\theta)R_i(\alpha)+\varepsilon f'_i(\theta) h_1(\alpha))(1-\cos(\theta-\alpha))}{A(\theta, \alpha)+\varepsilon B(f_i,\theta,\alpha)} \tilde{g}_{i,\varepsilon}(\alpha)d\alpha &\\
\qquad\qquad\,-\frac{1}{R_i(\theta)}P.V.\int\!\!\!\!\!\!\!\!\!\; {}-{} \frac{ f'_i(\theta)R_i(\alpha)\tilde{g}_{i,\varepsilon}(\alpha)(1-\cos(\theta-\alpha))}{(A(\theta, \alpha)+\varepsilon B(f_i,\theta,\alpha))^2} \left[4(h_1(\theta)+h_1(\alpha))\sin^2\left(\frac{\theta-\alpha}{2}\right)\right.&\\
\qquad\qquad\quad\, \left.+2\varepsilon (f_i(\theta)-f_i(\alpha))(h_1(\theta)-h_1(\alpha))+4\varepsilon\left(h_1(\theta)f_i(\alpha)+h_1(\alpha)f_i(\theta)\right)\sin^2\left(\frac{\theta-\alpha}{2}\right)\right]d\alpha\\
\qquad\qquad\,-\frac{\varepsilon h_1(\theta)}{R_i(\theta)^2}P.V.\int\!\!\!\!\!\!\!\!\!\; {}-{} \frac{\varepsilon f'_i(\theta)(f_i(\theta)-f_i(\alpha))}{A(\theta, \alpha)+\varepsilon B(f_i,\theta,\alpha)} \tilde{g}_{i,\varepsilon}(\alpha)d\alpha &\\
\qquad\qquad\,+\frac{1}{R_i(\theta)}P.V.\int\!\!\!\!\!\!\!\!\!\; {}-{} \frac{\varepsilon (h'_1(\theta)(f_i(\theta)-f_i(\alpha))+f'_i(\theta)(h_1(\theta)-h_1(\alpha)))}{A(\theta, \alpha)+\varepsilon B(f_i,\theta,\alpha)} \tilde{g}_{i,\varepsilon}(\alpha)d\alpha &\\
	\qquad\qquad\,-\frac{1}{R_i(\theta)}P.V.\int\!\!\!\!\!\!\!\!\!\; {}-{} \frac{\varepsilon f'_i(\theta)(f_i(\theta)-f_i(\alpha))\tilde{g}_{i,\varepsilon}(\alpha)}{(A(\theta, \alpha)+\varepsilon B(f_i,\theta,\alpha))^2}\left[4(h_1(\theta)+h_1(\alpha))\sin^2\left(\frac{\theta-\alpha}{2}\right)\right.&\\
\qquad\qquad\quad\, \left.+2\varepsilon (f_i(\theta)-f_i(\alpha))(h_1(\theta)-h_1(\alpha))+4\varepsilon\left(h_1(\theta)f_i(\alpha)
+h_1(\alpha)f_i(\theta)\right)\sin^2\left(\frac{\theta-\alpha}{2}\right)\right]d\alpha &\\
\qquad\qquad\,\,+O(\varepsilon),&\\
		\end{array}
	\end{equation}
	
	\begin{equation}\label{3-15}
		\partial_{f_j}\mathcal{F}_{i,1}(\varepsilon, \mathbf{x}, \bm f, \bm g) h_1=O(\varepsilon),
	\end{equation}
	
	\begin{equation}\label{3-16}
		\begin{array}{ll}
\partial_{g_i}\mathcal{F}_{i,1}( \varepsilon,\mathbf{x}, \bm f, \bm g)h_2
			=P.V.\int\!\!\!\!\!\!\!\!\!\; {}-{} \frac{\tilde{f}_{i,\varepsilon}(\alpha)\sin(\theta-\alpha)}{A(\theta, \alpha)+\varepsilon B(f_i,\theta,\alpha)}  h_2(\alpha)d\alpha\,&\\
	\quad\quad\,		+\frac{\varepsilon}{R_i(\theta)}P.V.\int\!\!\!\!\!\!\!\!\!\; {}-{} \frac{ f'_i(\theta)R_i(\alpha)(1-\cos(\theta-\alpha))}{A(\theta, \alpha)+\varepsilon B(f_i,\theta,\alpha)} h_2(\alpha)d\alpha+\frac{\varepsilon}{R_i(\theta)}P.V.\int\!\!\!\!\!\!\!\!\!\; {}-{} \frac{\varepsilon f'_i(\theta)(f_i(\theta)-f_i(\alpha))}{A(\theta, \alpha)+\varepsilon B(f_i,\theta,\alpha)} h_2(\alpha)d\alpha\,&\\
	\quad\quad\,-\frac{2\pi\varepsilon}{R_i(\theta)}\int\!\!\!\!\!\!\!\!\!\; {}-{} \nabla H(z_i(\theta),z_i(\alpha))\cdot (R_i(\theta)(-\sin\theta, \cos\theta)+\varepsilon f'_i(\theta)(\cos\theta, \sin\theta))h_2(\alpha) d\alpha,\,&\\
		\end{array}
	\end{equation}
	
	\begin{equation}\label{3-17}
		\begin{split}
			&\partial_{g_j}\mathcal{F}_{i,1}( \varepsilon,\mathbf{x}, \bm f, \bm g)h_2\\
			=&\frac{\varepsilon}{R_i(\theta)}\int\!\!\!\!\!\!\!\!\!\; {}-{} \frac{(x_{i}-x_{j})\cdot \left[R_i(\theta)(-\sin\theta, \cos\theta)+\varepsilon f'_i(\theta) (\cos\theta, \sin\theta)\right] }{A_{ij}+\varepsilon B_{ij}(\theta,\alpha)} h_2(\alpha)d\alpha\\
			+&\frac{\varepsilon}{R_i(\theta)}\int\!\!\!\!\!\!\!\!\!\; {}-{} \frac{\varepsilon^2f'_i(\theta)R_i(\theta)-\varepsilon^2f'_i(\theta)R_j(\alpha)\cos(\theta-\alpha)+ \varepsilon R_i(\theta)R_j(\alpha)\sin(\theta-\alpha)}{A_{ij}+\varepsilon B_{ij}(\theta,\alpha)} h_2(\alpha)d\alpha\\
			-&\frac{2\pi\varepsilon}{R_i(\theta)}\int\!\!\!\!\!\!\!\!\!\; {}-{} \nabla H(z_i(\theta),z_j(\alpha))\cdot [R_i(\theta)(-\sin\theta, \cos\theta)+\varepsilon f'_i(\theta)(\cos\theta, \sin\theta)]h_2(\alpha) d\alpha\\
			=&O(\varepsilon),
		\end{split}
	\end{equation}	

	\begin{equation}\label{3-18}
		\begin{split}
			&\partial_{f_i}\tilde{\mathcal{F}}_{i,2}( \varepsilon,\mathbf{x}, \bm f, \bm g) h_1\\
			=&-\frac{2 \tilde{g}_{i,\varepsilon}(\theta)(R_i(\theta)h_1(\theta)+\varepsilon f_i'(\theta)h_1'(\theta))}{(R_i(\theta)^2+(R_i'(\theta))^2)^2}P.V.\int\!\!\!\!\!\!\!\!\!\; {}-{} \frac{\varepsilon f_i'(\theta)R_i(\alpha)\sin(\theta-\alpha)}{A(\theta, \alpha)+\varepsilon B(f_i,\theta,\alpha)} \tilde{g}_{i,\varepsilon}(\alpha)d\alpha\\ +&\frac{\tilde{g}_{i,\varepsilon}(\theta)}{R_i(\theta)^2+(R'_i(\theta))^2}P.V.\int\!\!\!\!\!\!\!\!\!\; {}-{} \frac{ (h_1'(\theta)R_i(\alpha)+\varepsilon f_i'(\theta) h_1(\alpha))\sin(\theta-\alpha)}{A(\theta, \alpha)+\varepsilon B(f_i,\theta,\alpha)} \tilde{g}_{i,\varepsilon}(\alpha)d\alpha\\ -&\frac{\tilde{g}_{i,\varepsilon}(\theta)}{R_i(\theta)^2+(R_i'(\theta))^2}P.V.\int\!\!\!\!\!\!\!\!\!\; {}-{} \frac{ f_i'(\theta)R_i(\alpha)\tilde{g}_{i,\varepsilon}(\alpha)\sin(\theta-\alpha)}{(A(\theta, \alpha)+\varepsilon B(f_i,\theta,\alpha))^2} \left[4(h_1(\theta)+h_1(\alpha))\sin^2\left(\frac{\theta-\alpha}{2}\right)\right.\\
			&\quad \left.+2\varepsilon (f_i(\theta)-f_i(\alpha))(h_1(\theta)-h_1(\alpha))+4\varepsilon\left(h_1(\theta)f_i(\alpha)+h_1(\alpha)f_i(\theta)\right)\sin^2\left(\frac{\theta-\alpha}{2}\right)\right]d\alpha\\
			-&\frac{2 \tilde{g}_{i,\varepsilon}(\theta)(R_i(\theta)h_1(\theta)+\varepsilon f_i'(\theta)h_1'(\theta))}{(R_i(\theta)^2+(R_i'(\theta))^2)^2}P.V.\int\!\!\!\!\!\!\!\!\!\; {}-{} \frac{R_i(\theta)R_i(\alpha)(\cos(\theta-\alpha)-1)}{A(\theta, \alpha)+\varepsilon B(f_i,\theta,\alpha)} \tilde{g}_{i,\varepsilon}(\alpha)d\alpha\\
			 +&\frac{\tilde{g}_{i,\varepsilon}(\theta)}{R_i(\theta)^2+(R_i'(\theta))^2}P.V.\int\!\!\!\!\!\!\!\!\!\; {}-{} \frac{( h_1(\theta)R_i(\alpha)+R_i(\theta) h_1(\alpha))(\cos(\theta-\alpha)-1)}{A(\theta, \alpha)+\varepsilon B(f_i,\theta,\alpha)} \tilde{g}_{i,\varepsilon}(\alpha)d\alpha\\
			 -&\frac{\tilde{g}_{i,\varepsilon}(\theta)}{\varepsilon(R_i(\theta)^2+(R_i'(\theta))^2)}P.V.\int\!\!\!\!\!\!\!\!\!\; {}-{} \frac{R_i(\theta)R_i(\alpha)\tilde{g}_{i,\varepsilon}(\alpha)(\cos(\theta-\alpha)-1)}{(A(\theta, \alpha)+\varepsilon B(f_i,\theta,\alpha))^2}\\ &\quad\times\left[4(h_1(\theta)+h_1(\alpha))\sin^2\left(\frac{\theta-\alpha}{2}\right)+2\varepsilon (f_i(\theta)-f_i(\alpha))(h_1(\theta)-h_1(\alpha))\right.\\
			 &\qquad\left.+4\varepsilon\left(h_1(\theta)f_i(\alpha)+h_1(\alpha)f_i(\theta)\right)\sin^2\left(\frac{\theta-\alpha}{2}\right)\right]d\alpha\\
			-&\frac{2 \tilde{g}_{i,\varepsilon}(\theta)(R_i(\theta)h_1(\theta)+\varepsilon f_i'(\theta)h'_1(\theta))}{(R_i(\theta)^2+(R_i'(\theta))^2)^2}P.V.\int\!\!\!\!\!\!\!\!\!\; {}-{} \frac{R_i(\theta)( f_i(\alpha)-f_i(\theta))}{A(\theta, \alpha)+\varepsilon B(f_i,\theta,\alpha)} \tilde{g}_{i,\varepsilon}(\alpha)d\alpha\\
			 +&\frac{\tilde{g}_{i,\varepsilon}(\theta)}{R_i(\theta)^2+(R'_i(\theta))^2}P.V.\int\!\!\!\!\!\!\!\!\!\; {}-{} \frac{\varepsilon h_1(\theta)( f_i(\alpha)-f_i(\theta))+R_i(\theta)( h_1(\alpha)-h_1(\theta))}{A(\theta, \alpha)+\varepsilon B(f_i,\theta,\alpha)} \tilde{g}_{i,\varepsilon}(\alpha)d\alpha\\
			 -&\frac{\tilde{g}_{i,\varepsilon}(\theta)}{R_i(\theta)^2+(R'_i(\theta))^2}P.V.\int\!\!\!\!\!\!\!\!\!\; {}-{} \frac{R_i(\theta)( f_i(\alpha)-f_i(\theta))\tilde{g}_{i,\varepsilon}(\alpha)}{(A(\theta, \alpha)+\varepsilon B(f_i,\theta,\alpha))^2} \left[4(h_1(\theta)+h_1(\alpha))\sin^2\left(\frac{\theta-\alpha}{2}\right)\right.\\
			&\quad \left.+2\varepsilon (f_i(\theta)-f_i(\alpha))(h_1(\theta)-h_1(\alpha))+4\varepsilon\left(h_1(\theta)f_i(\alpha)+h_1(\alpha)f_i(\theta)\right)\sin^2\left(\frac{\theta-\alpha}{2}\right)\right]d\alpha\\
			&+O(\varepsilon),
		\end{split}
	\end{equation}
	
	\begin{equation}\label{3-19}
		\partial_{f_j}\mathcal{F}_{i, 2}(\varepsilon,\mathbf{x}, \bm f, \bm g) h_1=O(\varepsilon),
	\end{equation}
	
	\begin{equation}\label{3-20}
		\begin{split}
			&\partial_{g_i}\tilde{\mathcal{F}}_{i,2}( \varepsilon,\mathbf{x}, \bm f,\bm  g)h_2\\
			=&\frac{ h_2(\theta)}{R_i(\theta)^2+(R_i'(\theta))^2}P.V.\int\!\!\!\!\!\!\!\!\!\; {}-{} \frac{\varepsilon f_i'(\theta)R_i(\alpha)\sin(\theta-\alpha)}{A(\theta, \alpha)+\varepsilon B(f_i,\theta,\alpha)} \tilde{g}_{i,\varepsilon}(\alpha)d\alpha\\
			 +&\frac{{\tilde{g}_{i,\varepsilon}(\theta)}}{R_i(\theta)^2+(R_i'(\theta))^2}P.V.\int\!\!\!\!\!\!\!\!\!\; {}-{} \frac{\varepsilon f_i'(\theta)R_i(\alpha)\sin(\theta-\alpha)}{A(\theta, \alpha)+\varepsilon B(f_i,\theta,\alpha)}  h_2(\alpha)d\alpha\\
			+&\frac{ h_2(\theta)}{R(\theta)^2+(R'(\theta))^2}P.V.\int\!\!\!\!\!\!\!\!\!\; {}-{} \frac{R_i(\theta)R_i(\alpha)(\cos(\theta-\alpha)-1)}{A(\theta, \alpha)+\varepsilon B(f_i,\theta,\alpha)} \tilde{g}_{i,\varepsilon}(\alpha)d\alpha\\
			 +&\frac{\tilde{g}_{i,\varepsilon}(\theta)}{R_i(\theta)^2+(R_i'(\theta))^2}P.V.\int\!\!\!\!\!\!\!\!\!\; {}-{} \frac{R_i(\theta)R_i(\alpha)(\cos(\theta-\alpha)-1)}{A(\theta, \alpha)+\varepsilon B(f_i,\theta,\alpha)}  h_2(\alpha)d\alpha\\
			+&\frac{\varepsilon h_2(\theta)}{R_i(\theta)^2+(R_i'(\theta))^2}P.V.\int\!\!\!\!\!\!\!\!\!\; {}-{} \frac{R_i(\theta)( f_i(\alpha)-f_i(\theta))}{A(\theta, \alpha)+\varepsilon B(f_i,\theta,\alpha)} \tilde{g}_{i,\varepsilon}(\alpha)d\alpha\\ +&\frac{\varepsilon\tilde{g}_{i,\varepsilon}(\theta)}{R_i(\theta)^2+(R_i'(\theta))^2}P.V.\int\!\!\!\!\!\!\!\!\!\; {}-{} \frac{R_i(\theta)( f_i(\alpha)-f_i(\theta))}{A(\theta, \alpha)+\varepsilon B(f_i,\theta,\alpha)} h_2(\alpha)d\alpha\\
			+&\sum_{j\not=i}\frac{\varepsilon h_2(\theta)}{R_i(\theta)^2+(R_i'(\theta))^2}\int\!\!\!\!\!\!\!\!\!\; {}-{} \frac{(x_{i,\varepsilon}-x_{j,\varepsilon})^\perp\cdot \left[R_i(\theta)(-\sin\theta, \cos\theta)+\varepsilon f'_i(\theta) (\cos\theta, \sin\theta)\right] }{A_{ij}+\varepsilon B_{ij}(\theta,\alpha)} \tilde{g}_{j,\varepsilon}(\alpha)d\alpha\\
			+&\sum_{j\not=i}\frac{\varepsilon h_2(\theta)}{R_i(\theta)^2+(R_i'(\theta))^2}\int\!\!\!\!\!\!\!\!\!\; {}-{} \frac{-\varepsilon R_i^2(\theta)+\varepsilon^2f'_i(\theta)R_j(\alpha)\sin(\theta-\alpha)+ \varepsilon R_i(\theta)R_j(\alpha)\cos(\theta-\alpha)}{A_{ij}+\varepsilon B_{ij}(\theta,\alpha)} \tilde{g}_{j,\varepsilon}(\alpha)d\alpha\\
			-&\sum_{j=1}^m  \frac{2\pi\varepsilon h_2(\theta)}{R_i(\theta)^2+(R_i'(\theta))^2}\int\!\!\!\!\!\!\!\!\!\; {}-{} \nabla^\perp H(z_i(\theta),z_j(\alpha))\cdot (R_i(\theta)(-\sin\theta, \cos\theta)+\varepsilon f'_i(\theta)(\cos\theta, \sin\theta))\\
			&\quad\times\tilde{g}_{j,\varepsilon}(\alpha) d\alpha\\
			-& \frac{2\pi\varepsilon\tilde{g}_{i,\varepsilon}(\theta)}{R_i(\theta)^2+(R_i'(\theta))^2}\int\!\!\!\!\!\!\!\!\!\; {}-{} \nabla^\perp H(z_i(\theta),z_i(\alpha))\cdot (R_i(\theta)(-\sin\theta, \cos\theta)+\varepsilon f'_i(\theta)(\cos\theta, \sin\theta)) h_2(\alpha) d\alpha,
		\end{split}
	\end{equation}

	and
	
	\begin{equation}\label{3-21}
		\begin{array}{ll}
\partial_{g_j}\tilde{\mathcal{F}}_{i,2}( \varepsilon,\mathbf{x}, \bm f, \bm g)h_2\,&\\
			 \qquad\,=\frac{\varepsilon\tilde{g}_{i,\varepsilon}(\theta)}{R_i(\theta)^2+(R_i'(\theta))^2}\Big\{\int\!\!\!\!\!\!\!\!\!\; {}-{} \frac{(x_{i}-x_{j})^\perp\cdot \left[R_i(\theta)(-\sin\theta, \cos\theta)+\varepsilon f'_i(\theta) (\cos\theta, \sin\theta)\right] }{A_{ij}+\varepsilon B_{ij}(\theta,\alpha)} h_2(\alpha)d\alpha\,&\\
\qquad\quad\,+\int\!\!\!\!\!\!\!\!\!\; {}-{} \frac{-\varepsilon R_i^2(\theta)+\varepsilon^2f'_i(\theta)R_j(\alpha)\sin(\theta-\alpha)+ \varepsilon R_i(\theta)R_j(\alpha)\cos(\theta-\alpha)}{A_{ij}+\varepsilon B_{ij}(\theta,\alpha)}h_2(\alpha)d\alpha\,&\\
\qquad\quad\,- 2\pi\int\!\!\!\!\!\!\!\!\!\; {}-{} \nabla^\perp H(z_i(\theta),z_j(\alpha))\cdot [R_i(\theta)(-\sin\theta, \cos\theta)+\varepsilon f'_i(\theta)(\cos\theta, \sin\theta)]h_2(\alpha) d\alpha\Big\}\,&\\
\qquad\,=O(\varepsilon).\,&\\
		\end{array}
	\end{equation}		
\end{proof}

\section{Linearization and isomorphism}
In this section, we study the linearization of the functionals defined in Section \ref{2}. Denote $\bm{\mathcal{F}}_i:=(\mathcal{F}_{i,1}, \mathcal{F}_{i,2})$ and $\bm{\mathcal{F}}:=(\bm{\mathcal{F}}_1,\ldots, \bm{\mathcal{F}}_m)$.

By \eqref{3-8} and \eqref{3-10}, one can check that $(0,\mathbf{x}, 0,0)$ is a solution to $\bm {\mathcal{F}}=0$ if and only if $\mathbf{x}$ is a critical point of $\mathcal{W}_m$. Now, we take $\mathbf{x}_0$ to be a critical point of $\mathcal{W}_m$, and hence $(0,\mathbf{x}_0, 0,0)$ is a solution to $\bm{\mathcal{F}}=0$. We study the linearization of $\bm{\mathcal{F}}$ at $(0,\mathbf{x}_0, 0,0)$.

According to \eqref{3-14}-\eqref{3-21} at the end of the proof of Proposition \ref{p3-2}, when $\varepsilon=0$ and $\bm f, \bm g\equiv 0$, for all $i=1,\ldots,m$, the Gateaux derivatives are
\begin{equation}\label{4-1}
\left\{
	\begin{split}
		&\partial_{f_i}\mathcal{F}_{i,1}(0,\mathbf{x},  0,0) f=\frac{\kappa_i}{2}f'(\theta),\\
		&\partial_{f_j}\mathcal{F}_{i,1}(0,\mathbf{x},  0,0) f=0,\,\,\,j\not=i\\
		&\partial_{g_i}\mathcal{F}_{i,1}(0,\mathbf{x},  0,0) g=\int\!\!\!\!\!\!\!\!\!\; {}-{}\frac{g(\alpha)\sin(\theta-\alpha)}{4\sin^2\left(\frac{\theta-\alpha}{2}\right)}d\alpha,\\
		&\partial_{g_j}\mathcal{F}_{i,1}(0,\mathbf{x},  0,0) g=0,\,\,\,j\not=i\\
		&\partial_{f_i}\mathcal{F}_{i,2}(0,\mathbf{x},  0,0) f=\kappa_i^2f(\theta)-\kappa_i^2\int\!\!\!\!\!\!\!\!\!\; {}-{}\frac{f(\theta)-f(\alpha)}{4\sin^2\left(\frac{\theta-\alpha}{2}\right)}d\alpha,\\
		&\partial_{f_j}\mathcal{F}_{i,2}(0,\mathbf{x},  0,0) f=0,\,\,\,j\not=i\\
		&\partial_{g_i}\mathcal{F}_{i,2}(0,\mathbf{x},  0,0) g=-\frac{\kappa_i}{2}g(\theta),\\
		&\partial_{g_j}\mathcal{F}_{i,2}(0,\mathbf{x},  0,0) g=0\,\,\,j\not=i.
	\end{split}	
\right.
\end{equation}

Taking $(h_1,h_2)\in X^{k+1}\times X^k$, where
\begin{equation}\label{4-2}
	h_1(\theta)=\sum_{j=1}^\infty( a_j \cos(j\theta)+b_j\sin(j\theta)) \ \ \ \text{and} \ \ \ h_2(\theta)=\sum_{j=1}^\infty( c_j \cos(j\theta)+d_j\sin(j\theta)),
\end{equation}
we will prove that the linearization of $\bm {\mathcal{F}}_i$ at $(0,\mathbf{x}_0,0,0)$ has the following Fourier series form
\begin{equation}\label{4-3}
	\begin{split}
	 D_{(f_i,g_i)} \bm {\mathcal{F}}_i(0,\mathbf{x}_0,0,0) (h_1, h_2)=&\begin{pmatrix}
		\partial_{f_i}\mathcal{F}_{i,1}(0,\mathbf{x}_0,  0,0) h_1 + \partial_{g_i}\mathcal{F}_{i,1}(0,\mathbf{x}_0,  0,0) h_2\\
		\partial_{f_i}\mathcal{F}_{i,2}(0,\mathbf{x}_0,  0,0) h_1 + \partial_{g_i}\mathcal{F}_{i,2}(0,\mathbf{x}_0,  0,0) h_2
	\end{pmatrix}\\
	=&\sum_{j=1}^\infty\begin{pmatrix}  \hat a_j \sin(j\theta) +\hat b_j \cos(j\theta)\\  \hat c_j \cos(j\theta)+\hat d_j \sin(j\theta)\end{pmatrix},
	\end{split}
\end{equation}
where
\begin{equation*}
	\begin{pmatrix} \hat a_j  \\ \hat c_j \end{pmatrix}=M_j  \begin{pmatrix} a_j  \\  c_j \end{pmatrix} \quad\text{and}\,\,\begin{pmatrix} \hat b_j  \\ \hat d_j \end{pmatrix}=N_j  \begin{pmatrix} b_j  \\  d_j \end{pmatrix}
\end{equation*}
with $M_j$ and $N_j$ two $2\times 2$ matrices  given in Lemma \ref{l4-2}.

To compute $M_j$ and $N_j$, we need the following identities.

\begin{lemma}\label{l4-1}
	 For all $j\ge1$ and $j\in \mathbb N^*$, there hold
	\begin{equation}\label{4-4}
		\int\!\!\!\!\!\!\!\!\!\; {}-{}\frac{\cos(j\alpha)\sin(\theta-\alpha)}{4\sin^2\left(\frac{\theta-\alpha}{2}\right)}d\alpha
=\frac{1}{2}\sin(j\theta),
	\end{equation}
	\begin{equation}\label{4-5}
		\int\!\!\!\!\!\!\!\!\!\; {}-{}\frac{\sin(j\alpha)\sin(\theta-\alpha)}{4\sin^2\left(\frac{\theta-\alpha}{2}\right)}d\alpha
=-\frac{1}{2}\cos(j\theta),
	\end{equation}
	\begin{equation}\label{4-6}
		\int\!\!\!\!\!\!\!\!\!\; {}-{}\frac{\cos(j\theta)-\cos(j\alpha)}{4\sin^2\left(\frac{\theta-\alpha}{2}\right)}d\alpha
=\frac{j}{2}\cos(j\theta),
	\end{equation}
	\begin{equation}\label{4-7}
		\int\!\!\!\!\!\!\!\!\!\; {}-{}\frac{\sin(j\theta)-\sin(j\alpha)}{4\sin^2\left(\frac{\theta-\alpha}{2}\right)}d\alpha
=\frac{j}{2}\sin(j\theta).
	\end{equation}
\end{lemma}
\begin{proof}
	Identities \eqref{4-4} and \eqref{4-6} were already proved in Lemma A.8.  \cite{GPSY2}. Indeed, \eqref{4-4} can be deduced from the identity
	\begin{equation*}
		\int\!\!\!\!\!\!\!\!\!\; {}-{}\frac{\cos(j\alpha)\sin(\theta-\alpha)}{4\sin^2\left(\frac{\theta-\alpha}{2}\right)}d\alpha=\frac{1}{2}\int\!\!\!\!\!\!\!\!\!\; {}-{}\cos(j\alpha)\cot\left(\frac{\theta-\alpha}{2}\right)d\alpha=\frac{1}{2} \mathcal H(\cos(j\theta))(\theta),
	\end{equation*}
	where $\mathcal H(\cdot)$ is the Hilbert transform on torus and hence $H(\cos(j\theta))=\sin({j\theta})$.
	\eqref{4-6} can be obtained by computing the fractional Laplacians
	\begin{equation*}
		\int\!\!\!\!\!\!\!\!\!\; {}-{}\frac{\cos(j\theta)-\cos(j\alpha)}{4\sin^2\left(\frac{\theta-\alpha}{2}\right)}d\alpha=\frac{1}{2}(-\Delta)^{\frac{1}{2}}\cos(j\theta)=\frac{j}{2}\cos(j\theta).
	\end{equation*}
Finally, we point out that the identities \eqref{4-5} and \eqref{4-7} can be derived by calculating derivatives of \eqref{4-4} and \eqref{4-6} respectively.

\end{proof}

Now, we can prove \eqref{4-3} and find the explicit formula for $M_j$ and $N_j$.

\begin{lemma}\label{l4-2}
	The derivative of $\bm{\mathcal{F}}_i$ at $(0,\mathbf{x}_0,0,0)$ is given by \eqref{4-3} with
	\begin{equation}\label{4-8}
		M_j=\begin{pmatrix} - \frac{\kappa_i j}{2} & \frac{1}{2} \\ \frac{(2-j)\kappa_i^2}{2} & -\frac{\kappa_i}{2} \end{pmatrix}, \quad N_j=\begin{pmatrix}  \frac{\kappa_i j}{2} & -\frac{1}{2} \\ \frac{(2-j)\kappa_i^2}{2} & -\frac{\kappa_i}{2} \end{pmatrix}
	\end{equation}
    for any $j\geq 1$.

	Moreover, $D_{(f_i, g_i)}\bm{\mathcal{F}}_i(0,\mathbf{x}_0,0,0)$ is an isomorphism from $X^k_i$ to $Y^k_i$ and $D_{(\bm f, \bm g)}\bm{\mathcal{F}}(0,\mathbf{x}_0,0,0)$ is an isomorphism from $\mathcal X^k$ to $\mathcal Y^k$.
\end{lemma}
\begin{proof}
	Using \eqref{4-1}, \eqref{4-2} and Lemma \ref{l4-1}, we obtain by direct calculations
	\begin{equation*}
		\partial_{f_i}\mathcal{F}_{i,1}(0,\mathbf{x}_0,  0,0) h_1=\sum_{j=1}^\infty\left(\frac{-\kappa_ij}{2} a_j \sin(j\theta)+\frac{\kappa_ij}{2} b_j \cos(j\theta)\right),
	\end{equation*}
	\begin{equation*}
		\partial_{g_i}\mathcal{F}_{i,1}(0,\mathbf{x}_0,  0,0) h_2=\frac{1}{2}\sum_{j=1}^\infty (c_j\sin(j\theta)-d_j\cos(j\theta)),
	\end{equation*}		
	
	\begin{equation*}
		\begin{array}{ll}
\partial_{f_i}\mathcal{F}_{i,2}(0,\mathbf{x}_0,  0,0)h_1&\\
\qquad\qquad\,=\kappa_i^2\sum_{j=1}^\infty( a_j \cos(j\theta)+b_j\sin(j\theta))-\kappa_i^2\sum_{j=1}^\infty\left( \frac{j}{2} a_j \cos(j\theta)+\frac{j}{2}b_j\sin(j\theta)\right)&\\
\qquad\qquad\,=\sum_{j=1}^\infty\left( \frac{\kappa_i^2(2-j)}{2} a_j \cos(j\theta)+\frac{\kappa_i^2(2-j)}{2}b_j\sin(j\theta)\right),&\\
		\end{array}	
	\end{equation*}
	and
	\begin{equation*}
		\partial_{g_i}\mathcal{F}_{i,2}(0,\mathbf{x}_0,  0,0)h_2=-\frac{\kappa_i}{2}\sum_{j=1}^\infty( c_j \cos(j\theta)+d_j\sin(j\theta)).
	\end{equation*}	
	Then, one can easily check that the derivative of $\bm{\mathcal{F}}_i$ at $(0,\mathbf{x},0,0)$ is given by \eqref{4-3} with
	$$M_j=\begin{pmatrix} - \frac{\kappa_i j}{2} & \frac{1}{2} \\ \frac{(2-j)\kappa_i^2}{2} & -\frac{\kappa_i}{2} \end{pmatrix},\,\,N_j=\begin{pmatrix}  \frac{\kappa_i j}{2} & -\frac{1}{2} \\ \frac{(2-j)\kappa_i^2}{2} & -\frac{\kappa_i}{2} \end{pmatrix}.$$
	
	Now we are going to prove that $D_{(f_i,g_i)} \bm{\mathcal{F}}_i(0,\mathbf{x}_0,0,0)$ is an isomorphism from $X^k_i$ to $Y^k_i$. Recall the definition of $X^k_i$ and $Y^k_i$ given at the end of Section \ref{2}. From the above calculations, one has  $M_1=\begin{pmatrix} -\kappa_i/2 & 1/2 \\ \kappa_i^2/2 & -\kappa_i/2 \end{pmatrix}$ and $N_1=\begin{pmatrix} \kappa_i/2 & -1/2 \\ \kappa_i^2/2 & -\kappa_i/2 \end{pmatrix}$, then it is obvious that $D_{(f_i,g_i)} \bm{\mathcal{F}}_i(0,\mathbf{x}_0,0,0)$ maps $X^k_i$ to $Y^k_i$. Hence only the invertibility needs to be considered.
	
	For $j\ge 2$, $\det(M_j)=-\det(N_j)=\frac{\kappa_i^2(j-1)}{2}>0$ which implies that $M_j$ and $N_j$ are invertible, and their inverse are given by
	
	\begin{equation}\label{4-9}
		M_j^{-1}=\begin{pmatrix} \frac{-1}{\kappa_i(j-1)} & \frac{-1}{\kappa_i^2(j-1)} \\ \frac{j-2}{j-1} & \frac{-j}{\kappa_i(j-1)} \end{pmatrix}, \ \ \ \forall \, j\ge 2.
	\end{equation}
	and
	\begin{equation}\label{4-10}
		N_j^{-1}=\begin{pmatrix} \frac{1}{\kappa_i(j-1)} & \frac{-1}{\kappa_i^2(j-1)} \\ \frac{2-j}{j-1} & \frac{-j}{\kappa_i(j-1)} \end{pmatrix}, \ \ \ \forall \, j\ge 2.
	\end{equation}
	Thus for any $(u,v)\in  Y_i^k$ with
	\begin{equation*}
		u=\sum_{j=1}^\infty p_j\sin(j\theta) +q_j\cos(j\theta)\ \ \ \text{and} \ \ \ v=-\kappa_i p_1\cos(\theta)+\kappa_iq_1\sin(\theta)+\sum_{j=2}^\infty r_j\cos(j\theta)+s_j\sin(j\theta),
	\end{equation*}
	we can write $D_{(f_i,g_i)}\bm{\mathcal{F}}_i(0,\mathbf{x}_0,0,0)^{-1}(u,v)$ as
	\begin{equation*}
		D_{(f_i,g_i)}\bm{\mathcal{F}}_i(0,\mathbf{x}_0,0,0)^{-1}(u,v)=\begin{pmatrix}  -\frac{p_1}{\kappa_i} \cos(\theta)+\frac{q_1}{\kappa} \sin(\theta)\\  p_1 \cos(\theta)-q_1\sin(\theta)\end{pmatrix}+\sum_{j=2}^\infty M_j^{-1}  \begin{pmatrix} p_j  \\  r_j \end{pmatrix}\cos(j\theta)+ N_j^{-1}  \begin{pmatrix} q_j  \\  s_j \end{pmatrix}\sin(j\theta).
	\end{equation*}
	Denote
	\begin{equation*}
		\begin{pmatrix} \tilde p_j  \\ \tilde r_j \end{pmatrix}=M_j^{-1}  \begin{pmatrix} p_j  \\  r_j \end{pmatrix}, \ \ \	\begin{pmatrix} \tilde q_j  \\ \tilde s_j \end{pmatrix}=N_j^{-1}  \begin{pmatrix} q_j  \\  s_j \end{pmatrix},\ \ \  \forall \, j\ge 2.
	\end{equation*}
	From \eqref{4-9} and \eqref{4-10}, we have the asymptotic behavior: $\tilde p_j=O(j^{-1}(|p_j|+|r_j|))$, $\tilde r_j=O( |p_j|+|r_j|)$, $\tilde q_j=O(j^{-1}(|q_j|+|s_j|))$ and $\tilde s_j=O( |q_j|+|s_j|)$ as $j\to +\infty$, which implies that  $D_{(f_i,g_i)}\bm{\mathcal{F}}_i(0,\mathbf{x},0,0)^{-1}(u,v)$ does belong to $ X^k_i$.
	
	Notice that by \eqref{4-1}, we have  $\partial_{f_j}\mathcal{F}_{i,1}(0,\mathbf{x}_0,  0,0) h_1$, $\partial_{g_j}\mathcal{F}_{i,1}(0,\mathbf{x}_0,  0,0) h_2$, $\partial_{f_j}\mathcal{F}_{i,2}(0,\mathbf{x}_0,  0,0) h_1$, $\partial_{g_j}\mathcal{F}_{i,2}(0,\mathbf{x}_0,  0,0) h_2=0,\,j\not=i$. Therefore, we find $$D_{(\bm f,\bm g)}\bm{\mathcal{F}}(0,\mathbf{x}_0,  0,0)=\text{diag}\left(D_{(f_1,g_1)}\bm{\mathcal{F}}_1(0,\mathbf{x}_0,  0,0),\cdots, D_{(f_m,g_m)}\bm{\mathcal{F}}_m(0,\mathbf{x}_0,  0,0)\right),$$ and hence $D_{(\bm f,\bm g)}\bm{\mathcal{F}}(0,\mathbf{x}_0,  0,0)$ is an isomorphism from $\mathcal X^k$ to $\mathcal Y^k$.
	
	The proof of this lemma is completed.
	
\end{proof}

\section{Existence of vortex sheets}
In this section, inspired by the classical Crandall-Rabinowitz theorem on bifurcation theory, we use the implicit function theorem to obtain a branch of solutions for arbitrary fixed small $\varepsilon$.

From the previous sections, we know that $(0,\mathbf{x}_0, 0,0)$ is a solution to $\bm{\mathcal{F}}=0$ if and only if $\mathbf{x}_0$ is a critical point of $\mathcal{W}_m$.  Moreover,  $D_{(\bm f,\bm g)}\bm{\mathcal{F}}(0,\mathbf{x}_0, 0,0)$ is an isomorphism from $\mathcal X^k$ to $\mathcal Y^k$. It can be seen from Lemma \ref{l4-2} that the kernel of $D_{(\bm f, \bm g)}\bm{\mathcal{F}}(0,\mathbf{x}_0,0,0)$ in $\left(X^{k+1}\times X^k\right)^m$ is  $$\mathcal{X}^k_0:=\prod_{i=1}^m \text{span}\{(a_1\cos(\theta)+b_1\sin(\theta), \kappa_i(a_1\cos(\theta)+b_1\sin(\theta)))\}.$$
We take arbitrary nontrivial $(\bm f_0, \bm g_0)\in \mathcal{X}^k_0$ and define the following new functional
\begin{equation}\label{5-1}
	\overline{\bm{\mathcal{F}}}(\varepsilon, \tau, \mathbf{x}, \bm f, \bm g):= \bm{\mathcal{F}}(\varepsilon, \mathbf{x}, \bm f +\tau\bm{f}_0, \bm g+ \tau\bm{g}_0),
\end{equation}
where $\mathbf{x}_{\varepsilon,\tau}$ is closed to $\mathbf x$, the given critical point of $\mathcal{W}_m$, and will be determined later.

To apply the implicit function theorem, we need to make sure that $\overline{\bm{\mathcal{F}}}$ map from $\mathcal X^k$ to $\mathcal Y^k$. This will be achieved by careful choice of  $\mathbf{x}$. Indeed, taking $V_1\subset \mathcal X^k$ be the unit ball, we have the following key proposition.

\begin{proposition}\label{p5-1}
	The condition  that $\overline{\bm{\mathcal{F}}}$ maps from $(-\varepsilon_0, \varepsilon_0)\times(-\tau_1, \tau_1)\times B_{r_0}(\mathbf{x}_0)\times V_1$ to $\mathcal Y^k$ is equivalent to a equation of the form
	\begin{equation}\label{5-2}
		\nabla \mathcal{ W}_m(\mathbf{x})=O_{\tau_1}(\varepsilon),
	\end{equation}
	where $\tau_1$ is any fixed small positive number and $O_{\tau_1}(\varepsilon)$ means a vector that is of the order $\varepsilon$ up to a constant depending on $\tau_1$.
\end{proposition}
\begin{proof}
	For arbitrary $i=1,\ldots, m$, we take  $(\bm f, \bm g)\in V_1$  with
	\begin{equation*}
		\begin{split}
			&f_i(\theta)=\sum_{j=1}^\infty( a_j \cos(j\theta)+b_j\sin(j\theta)), \\  &g_i(\theta)=-\kappa_i a_1\cos(\theta)-\kappa_i b_1\sin(\theta)+\sum_{j=2}^\infty( c_j \cos(j\theta)+d_j\sin(j\theta)).
		\end{split}	
	\end{equation*}
	By the definition of $\mathcal Y^k$, in order to make $\overline{\bm{\mathcal{F}}}(\varepsilon, \tau, \mathbf{x}, \bm f, \bm g)= \bm{\mathcal{F}}(\varepsilon, \mathbf{x}, \bm f +\tau\bm{f}_0, \bm g+ \tau\bm{g}_0)\in \mathcal{Y}^k$, we need to ensure that the following equations hold true.
	\begin{equation}\label{5-3}
		\begin{split}
			&-\kappa_i \int\!\!\!\!\!\!\!\!\!\; {}-{}\mathcal{F}_{i,1}(\varepsilon, \mathbf{x}, \bm f +\tau\bm{f}_0, \bm g+ \tau\bm{g}_0) \sin(\theta)d\theta=\int\!\!\!\!\!\!\!\!\!\; {}-{}\mathcal{F}_{i,2}(\varepsilon, \mathbf{x}, \bm f +\tau\bm{f}_0, \bm g+ \tau\bm{g}_0) \cos(\theta)d\theta,\\
			&\kappa_i \int\!\!\!\!\!\!\!\!\!\; {}-{}\mathcal{F}_{i,1}(\varepsilon, \mathbf{x}, \bm f +\tau\bm{f}_0, \bm g+ \tau\bm{g}_0) \cos(\theta)d\theta=\int\!\!\!\!\!\!\!\!\!\; {}-{}\mathcal{F}_{i,2}(\varepsilon, \mathbf{x}, \bm f +\tau\bm{f}_0, \bm g+ \tau\bm{g}_0) \sin(\theta)d\theta,
		\end{split}
	\end{equation}
where $i=1,\ldots,m$.
	By \eqref{3-7}, \eqref{3-9} and calculations in Lemmas \ref{l4-1} and \ref{l4-2}, we obtain
	\begin{equation}\label{5-4}
		\begin{array}{ll}
			\int\!\!\!\!\!\!\!\!\!\; {}-{}\mathcal{F}_{i,1}(\varepsilon, \mathbf{x}, \bm f +\tau\bm{f}_0, \bm g+ \tau\bm{g}_0) \sin(\theta)d\theta
			=-\kappa_i a_1 -\sum_{j\not=i}2\pi \kappa_j\partial_{x_{i,1}} G(x_{i},x_{j})&\\
\qquad\qquad\qquad\qquad\qquad\qquad\,+ 2\pi  \kappa_i\partial_{x_{i,1}} H(x_{i},x_{i})
		+\varepsilon \int\!\!\!\!\!\!\!\!\!\; {}-{} \mathcal{R}_{1}\sin(\theta)d\theta,&\\
		\end{array}
\end{equation}
	
\begin{equation}\label{5-5}
\begin{array}{ll}
\int\!\!\!\!\!\!\!\!\!\; {}-{}\mathcal{F}_{i,1}(\varepsilon, \mathbf{x}, \bm f +\tau\bm{f}_0, \bm g+ \tau\bm{g}_0) \cos(\theta)d\theta =\kappa_i b_1
			+\sum_{j\not=i}2\pi \kappa_j \partial_{x_{i,2}} G(x_{i},x_{j})&\\
 \qquad\qquad\qquad\qquad\qquad\qquad\,- 2\pi  \kappa_i\partial_{x_{i,2}} H(x_{i},x_{i})
			+\varepsilon \int\!\!\!\!\!\!\!\!\!\; {}-{} \mathcal{R}_{1}\cos(\theta)d\theta,&\\
\end{array}
\end{equation}
	
\begin{equation}\label{5-6}
\begin{array}{ll}
\int\!\!\!\!\!\!\!\!\!\; {}-{}\mathcal{F}_{i,2}(\varepsilon, \mathbf{x}, \bm f +\tau\bm{f}_0, \bm g+ \tau\bm{g}_0) \sin(\theta)d\theta =\kappa_i^2 b_1
			-\sum_{j\not=i}2\pi \kappa_i \kappa_j \partial_{x_{i,2}} G(x_{i},x_{j})&\\
\qquad\qquad\qquad\qquad\qquad\qquad\,+ 2\pi \kappa_i^2\partial_{x_{i,2}} H(x_{i},x_{i})+\varepsilon \int\!\!\!\!\!\!\!\!\!\; {}-{} \mathcal{R}_{2}\sin(\theta)d\theta,&\\
\end{array}
\end{equation}	
	and	
	\begin{equation}\label{5-7}
		\begin{array}{ll}
\int\!\!\!\!\!\!\!\!\!\; {}-{}\mathcal{F}_{i,2}(\varepsilon, \mathbf{x}, \bm f +\tau\bm{f}_0, \bm g+ \tau\bm{g}_0) \cos(\theta)d\theta=\kappa_i^2a_1-\sum_{j\not=i}2\pi \kappa_i\kappa_j \partial_{x_{i,1}} G(x_{i},x_{j})&\\
\qquad\qquad\qquad\qquad\qquad\qquad\,+ 2\pi \kappa_i^2\partial_{x_{i,1}} H(x_{i},x_{i})+\varepsilon \int\!\!\!\!\!\!\!\!\!\; {}-{} \mathcal{R}_{2}\cos(\theta)d\theta.&\\
		\end{array}
	\end{equation}
	Then, by the above equations \eqref{5-4}-\eqref{5-7}, we conclude that \eqref{5-3} is equivalent to the following equations.
	
	\begin{equation}\label{5-8}
		\sum_{j\not=i} \kappa_i\kappa_j \partial_{x_{i,1}} G(x_{i},x_{j})-  \kappa_i^2\partial_{x_{i,1}} H(x_{i},x_{i})=\frac{\varepsilon}{4\pi}\left(\int\!\!\!\!\!\!\!\!\!\; {}-{} \mathcal{R}_{2}\cos(\theta)d\theta+\kappa_i\int\!\!\!\!\!\!\!\!\!\; {}-{} \mathcal{R}_{1}\sin(\theta)d\theta\right),
	\end{equation}
	and
	\begin{equation}\label{5-9}
		\sum_{j\not=i} \kappa_i\kappa_j \partial_{x_{i,2}} G(x_{i},x_{j}) -\kappa_i^2\partial_{x_{i,2}} H(x_{i},x_{i})=\frac{\varepsilon}{4\pi}\left(\int\!\!\!\!\!\!\!\!\!\; {}-{} \mathcal{R}_{2}\sin(\theta)d\theta-\kappa_i\int\!\!\!\!\!\!\!\!\!\; {}-{} \mathcal{R}_{1}\cos(\theta)d\theta\right).
	\end{equation}
	
	Since \eqref{5-8} and \eqref{5-9} hold for all $i=1,\ldots,m$, we arrive at \eqref{5-2} and complete the proof of this proposition.
	
\end{proof}

Now, we are ready to prove Theorem \ref{thm1}.

{\bf Proof of Theorem \ref{thm1}.}\,Since we have the nondegeneracy condition $\text{deg} \left(\nabla \mathcal{W}_m, \mathbf{x}_0\right)\not=0$, the equation \eqref{5-2} is solvable near $\mathbf{x}_0$ whenever $\varepsilon$ is small. We solve \eqref{5-2} and write the solution $\mathbf{x}_{\varepsilon,\tau}$ in the form $\mathbf{x}_{\varepsilon,\tau}=\mathbf{x}_0+\varepsilon  \bar{\mathcal{R}}_{\mathbf{x}}(\varepsilon, \tau, \bm f, \bm g)$. Then, we know that $\bar{\mathcal R}_{\mathbf{x}}$ defined on $ (-\varepsilon_0, \varepsilon_0) \times(-\tau_1, \tau_1) \times V_1$ is at least of $C^1$ smooth due to the regularity of $\bm {\mathcal{F}}$.

Now, set $$\overline{\bm {\mathcal{F}}}^*(\varepsilon,\tau, \bm f, \bm g):= \overline{\bm {\mathcal{F}}}(\varepsilon,\tau, \mathbf{x}_0+\varepsilon \bar{\mathcal R}_{\mathbf{x}}(\varepsilon, \tau, \bm f, \bm g), \bm f, \bm g).$$
Then, we conclude from Proposition \ref{p5-1} that $\overline{\bm {\mathcal{F}}}^*$ maps from $(-\varepsilon_0, \varepsilon_0)\times(-\tau_1, \tau_1) \times V_1$ to $\mathcal Y^k$. Moreover, $\overline{\bm {\mathcal{F}}}^*$ is $C^1$ continuous with respect to $\bm f$ and $\bm g$.  Next, we need to verify that $D_{(\bm f, \bm g)}\bm{\mathcal{F}}^*(0,0,0,0)$ is an isomorphism from $\mathcal X^k$ to $\mathcal Y^k$. In fact, by the chain rule we get
$$ D_{(\bm f, \bm g)} \overline{\bm {\mathcal{F}}}^*=D_{(\bm f, \bm g)} \overline{\bm {\mathcal{F}}}+ D_{\mathbf{x}}\overline{\bm {\mathcal{F}}}\cdot D_{(\bm f, \bm g)}\left(\mathbf{x}_0+\varepsilon \bar{\mathcal R}_{\mathbf{x}}(\varepsilon, \tau, \bm f, \bm g) \right),$$
which implies
$$ D_{(\bm f, \bm g)} \overline{\bm {\mathcal{F}}}^*(0,0,0,0)=D_{(\bm f, \bm g)} \bm{\mathcal{F}}(0, \mathbf{x}_0, 0,0).$$
Therefore, $D_{(\bm f, \bm g)} \overline{\bm {\mathcal{F}}}^*(0,0,0,0)$ is an isomorphism from $\mathcal X^k$ to $\mathcal Y^k$.

Now we  can apply implicit function theorem to $\overline{\bm {\mathcal{F}}}^*$ at the point $(0,0,0,0)$, and obtain  that there exist $\varepsilon_0>0$ and $0<\tau_0\leq \tau_1$ such that the solutions set
\begin{equation*}
	\left\{(\varepsilon, \tau, \bm f, \bm g)\in (-\varepsilon_0,\varepsilon_0)\times (-\tau_0,\tau_0)\times V_1 \ : \ \overline{\bm {\mathcal{F}}}^*(\varepsilon,\tau,\bm f, \bm g)=0\right\}
\end{equation*}
is not empty and can be parameterized by a two-dimensional surface $(\varepsilon,\tau)\in (-\varepsilon_0,\varepsilon_0)\times (-\tau_0,\tau_0)\to (\varepsilon, ,\tau, \bm{f}_{\varepsilon,\tau}, \bm{g}_{\varepsilon,\tau})$. So we  obtain a family of nontrivial vortex sheet solutions and finishes the proof of $(i)$ in Theorem \ref{thm1}. Since $(ii)$ of Theorem \ref{thm1}  is obvious, to end our proof we only need to show the convexity of the interior of $\Gamma_i$ for $i=1,\ldots, m$.
 This can be done by computing the sign of the curvature. Recall that $z_i(\theta)=x_i+R_i(\theta)(\cos\theta,\sin\theta)$ with $R_i(\theta)=1+\varepsilon(f_{\varepsilon,\tau,i}(\theta)+\tau f_{0,i})$. Given $\theta\in[0,2\pi)$, the signed curvature of $\Gamma_i$ at $z_i(\theta)$ is
\begin{align*}
	\varepsilon \kappa(\theta)=\frac{R_i(\theta)^2+2R'_i(\theta)^2-R_i(\theta)R_i''(\theta)}{\left(R_i(\theta)^2+R_i'(\theta)^2\right)^{\frac{3}{2}}}=\frac{1+O(\varepsilon)}{1+O(\varepsilon)}>0,
\end{align*}
for $\varepsilon$ and $\tau$ small, which implies the convexity and thus completes the proof of Theorem \ref{thm1}.

\qed

 We point out that for fixed $\varepsilon$, if $\tau_1\not=\tau_2$ with $0<\tau_1, \tau_2<\tau_0$, then obviously one has $\omega_{\varepsilon,\tau_1}\not=\omega_{\varepsilon,\tau_2}$. Thus, we have obtained a large family of stationary solutions with vortex sheet for every $\varepsilon>0$ small. Corollary \ref{thm2} follows immediately by taking $\tau=0$ in Theorem \ref{thm1}.

\phantom{s}
\thispagestyle{empty}

\end{document}